\documentclass[12pt]{article}
\usepackage{amsmath,amsthm,amsfonts,amssymb,verbatim}
\usepackage[usenames]{color} 

\newtheorem{thm}{Theorem}[section]
\newtheorem{lemma}[thm]{Lemma}
\newtheorem{cor}[thm]{Corollary}

\newtheorem{prop}[thm]{Proposition}
\newtheorem{question}[thm]{Question}

\usepackage{pdfsync}
\usepackage{graphicx}

\theoremstyle{definition}  

\numberwithin{equation}{section}
\newtheorem{defn}[thm]{Definition}

\newtheorem{notn}[thm]{Notation}

\newtheorem{ex}[thm]{Example}

\newtheorem{remark}[thm]{Remark}
\newtheorem{conj}[thm]{Conjecture}

\theoremstyle{definition}
\theoremstyle{remark}

\newcounter{enumitemp}

\newcommand\pref[1]{(\ref{#1})}

\DeclareMathOperator{\Fix}{Fix}

\DeclareMathOperator{\card}{card}
\DeclareMathOperator{\ent}{ent}

\DeclareMathOperator{\Diff}{Diff}
\DeclareMathOperator{\Df0}{Diff_0}
\DeclareMathOperator{\Symp}{Symp}

\DeclareMathOperator{\mcg}{MCG}

\DeclareMathOperator{\Int}{int}
\DeclareMathOperator{\Per}{Per}

\DeclareMathOperator{\fr}{fr}

\DeclareMathOperator{\cl}{cl}

\DeclareMathOperator{\zz}{{\cal Z}}

\newcommand{\cA}{{\cal A}}

\newcommand{\C}{{\cal C}}

\newcommand{\R}{\mathbb R}

\newcommand{\cS}{{\cal S}}
\newcommand{\cT}{{\cal T}}

\newcommand{\X}{{\chi}}
\newcommand{\cZ}{{\cal Z}}
\newcommand{\T}{\mathbb T}

\newcommand{\Z}{\mathbb Z}

\def\D{{\mathbb D}}

\def\ti{\tilde}

\def\sl3z{SL(3, \mathbb Z)}

\def\O{{\cal O}}
\def\G{{\cal G}}

\def\cS{{\cal S}}
\def\A{{\mathbb A}}

\newcommand{\trycomment}[1]{} 
 
\DeclareMathOperator{\Cent}{Cent}

\DeclareMathOperator{\Out}{Out}
\DeclareMathOperator{\Aut}{Aut}

\title{Some virtually abelian subgroups of 
the group of analytic symplectic diffeomorphisms of $S^2$
}
\author{John Franks,\thanks{Supported in part by NSF grant DMS0099640.}\ \ 
Michael Handel\thanks{Supported in part by NSF grant DMS0103435.}}

\begin{document}
\maketitle
\begin{abstract}
We show that if $M$ is a compact oriented surface of genus $0$ and $G$
is a subgroup of $\Symp^\omega_\mu(M)$ which has an infinite normal
solvable subgroup, then $G$ is virtually abelian.  In particular 
the centralizer of an infinite order $f \in \Symp^\omega_\mu(M)$
is virtually abelian. Another immediate corollary is that if $G$
is a solvable subgroup of $\Symp^\omega_\mu(M)$ then $G$ is virtually
abelian. We also prove a special case of the Tits Alternative 
for subgroups of  {$\Symp^\omega_\mu(M).$}
\end{abstract}
\section{Introduction}

If $M$ is a compact connected oriented surface we let $\Df0^r(M)$
denote the $C^r$ diffeomorphisms isotopic to the identity.  In
particular if $r = \omega$ this denotes the subgroup of those which
are real analytic. Let $\Symp^r_\mu(M)$ denote the subgroup of
$\Df0^r(M)$ which preserve a smooth volume form $\mu$.  We
will denote by $\Cent^r(f)$ and $\Cent^r_\mu(f)$ the subgroups
of $\Df0^r(M)$ and $\Symp^r_\mu(M)$, respectively, whose elements
commute with $f.$ If $G$ is a subgroup of $\Symp^r_\mu(M)$ then
$\Cent^r_\mu(f,G)$ will denote the the subgroup of $G$ whose
elements commute with $f.$  In this article we wish to address the
algebraic structure of subgroups of $\Symp^\infty_\mu(M)$ and our
results are largely limited to the case of analytic diffeomorphisms
when $M$ has genus $0$. Our approach is to understand the possible
dynamic behavior of such diffeomorphisms and use the structure
exhibited to conclude information about subgroups.

\begin{defn} 
If $N$ is a connected manifold,
an element $f \in \Diff^r(N)$  will be said to have {\em full support}
provided $N \setminus \Fix(f)$ is dense in $N$.  We will say that
$f$ has {\em support of finite type} if  both $\Fix(f)$ and 
$N \setminus \Fix(f)$ have finitely many components.
We will say that a subgroup $G$ of $\Diff^r(N)$ has {\em full support
of finite type} if every non-trivial element of $G$ 
has full support of finite type. 
\end{defn}

The primary case of interest for this article is analytic
diffeomorphisms, but we focus on groups with full support of finite
type to emphasize the properties we use, which are dynamical rather
than analytic in nature.  The following result shows we include
$\Diff^\omega(M)$.

\begin{prop}\label{analytic}
If $N$ is a compact connected manifold, the group 
$\Diff^\omega(N)$ has full support of finite type.
\end{prop}

\begin{proof}
 If $f$ is analytic and non-trivial the set $\Fix(f)$ 
is an analytic set in $N$ and has no interior and the set 
$N \setminus \Fix(f)$ is a semianalytic set.
Hence $\Fix(f)$ has finitely many components
 and  $N \setminus \Fix(f)$ has finitely many components 
(see Corollary 2.7 of \cite{B-M} for both these facts).
\end{proof}

The following result is due to Katok \cite{Katok2} who stated it only
in the analytic case. For completeness we give a proof in
section~\ref{sec:positive entropy} but the proof we give is
essentially the same as the analytic case.

\begin{prop}[Katok \cite{Katok2}]\label{prop cyclic}
Suppose $G$ is a subgroup of $\Diff^2(M)$ which has full support of
finite type and $f \in \Diff^2(M)$ has positive
topological entropy.  Then the centralizer of $f$
in $G$, is virtually cyclic.  Moreover, every infinite
order element of this centralizer has positive topological entropy.
\end{prop}

A corollary of the proof of this result is the following.

\begin{cor}\label{dichotomy}
Suppose $f,g \in \Diff^2(M)$ have infinite order and full support of
finite type. If $fg = gf$, then $f$ has positive topological entropy
if and only if $g$ has positive entropy.
\end{cor}

 We remark that in contrast to our subsequent results
  Proposition~\ref{prop cyclic} and its corollary make no assumption
  about preservation of a measure $\mu.$

We can now state our main result,  whose proof is given in section~\ref{main section}.

\begin{thm}\label{main thm}
 Suppose $M$ is a compact oriented surface of genus $0$
and  $G$ is a subgroup
 of $\Symp^\infty_\mu(M)$ which has full support of finite type, 
e.g. a subgroup of $\Symp^\omega_\mu(M)$.  Suppose further that $G$ has an infinite normal solvable subgroup. Then $G$ is virtually abelian.
\end{thm}

There are several interesting corollaries of this result. To motivate the first we recall the following result.

\begin{thm}[Farb-Shalen \cite{Farb-Shalen}]
Suppose $f \in \Df0^\omega(S^1)$ has 
infinite order.  Then the centralizer of $f$ in
$\Df0^\omega(S^1)$, is virtually abelian.
\end{thm}

 There are a number of instances where results about the algebraic
structure of $\Diff(S^1)$ have close parallels for $\Symp_\mu(M)$.
By analogy with the result of Farb and Shalen above, 
it is natural to ask the following question.

\begin{question}
Suppose $M$ is a compact surface and 
 $G$ is a subgroup of $\Df0^2(M)$ with full support of finite type
and $f \in G$  has infinite order. Is the centralizer of $f$ in $G$
always virtually abelian?
\end{question}

The following result gives a partial answer to this question.  

\begin{prop}\label{prop: cent}  
Suppose $M$ is a compact surface of genus $0$, $H$ 
is a subgroup of $\Symp^\infty_\mu(M)$ with full support of 
finite type and $f \in H$  
has infinite order, then $\Cent^\infty_\mu(f, H)$, the centralizer of $f$
in $H,$ is virtually abelian.  In particular this applies 
when $H \subset \Symp_\mu^\omega(M)$.
\end{prop}

\begin{proof}
If we apply Theorem~\ref{main thm} to the group $G = \Cent^\infty_\mu(f,H)$ and 
observe that the cyclic group generated by $f$ is a normal abelian subgroup,
then we conclude that $G$ is virtually abelian.
\end{proof}

 Question 1.9 of \cite{FH-ent0} asks if any finite index subgroup of
$\mcg(\Sigma_g)$ can act faithfully by area preserving diffeomorphisms
on a closed surface.  The following corollary of
Proposition~\ref{prop: cent} and Proposition~\ref{analytic} gives a
partial answer in the special case of analytic diffeomorphisms and $M$
of genus $0$.

\begin{cor}
Suppose $H$ is a finite index subgroup of $\mcg(\Sigma_g)$ with $g \ge 2$
or $\Out(F_n)$ with $n \ge 3$, and $G$ is a group
with the property that every element of infinite order has 
a virtually abelian centralizer.
Then any homomorphism $\nu: H \to G$ has
non-trivial kernel. In particular this holds  if $G = \Symp^\omega_\mu(S^2).$

\end{cor}
\begin{proof}  
Suppose at first that $H$ is a finite index subgroup of
$\mcg(\Sigma_g)$ with $g \ge 2$.  Let $\alpha_1, \alpha_2$ and
$\alpha_3$ be simple closed curves in $M$ such that $\alpha_1$ and
$\alpha_2$ are disjoint from $\alpha_3$ and intersect each other
transversely in exactly one point.  Let $T_i$ be a Dehn twist around
$\alpha_i$ with degree chosen so that $T_i \in H$.  Then $T_1$ and
$T_2$ commute with $T_3$ but no finite index subgroup of the group
generated by $T_1$ and $T_2$ is abelian.  It follows that $\mu$ cannot
be injective.

Suppose now that $H$ is a finite index subgroup of $\Out(F_n)$ with
$n\ge 3$.  Let $A,B$ and $C$ be three of the generators of $F_n$.
Define $\Phi \in \Aut(F_n)$ by $B \mapsto BA$, define $\Psi \in
\Aut(F_n)$ by $C \mapsto CBAB^{-1}A$ and define $\Theta \in \Aut(F_n)$
by $C \mapsto CABAB^{-1}$, where $\Phi$ fixes all generators other
than $B$ and where $\Psi $ and $\Theta$ fix all generators other than
$C$.  Let $\phi, \psi$ and $\theta$ be the elements of $\Out(F_n)$
determined by $\Phi, \Psi$ and $\Theta$ respectively.  It is
straightforward to check that $\Phi$ commutes with $\Psi$ and $\Theta$
and that for all $k > 0$, $\Psi^k$ (which is defined by $C \mapsto
C(BAB^{-1}A)^k$) does not commute with $\Theta^k$ (which is defined by
$C \mapsto C(ABAB^{-1})^k$).  Moreover, $[\Theta^k,\Psi^k]$ is not a
non-trivial inner automorphism because it fixes both $A$ and $B$.  It
follows that $[\theta^k,\psi^k]$ is non-trivial for all $k$ and hence
that no finite index subgroup of the group generated by $\psi$ and
$\theta$, and hence no finite index subgroup of $\Cent(\phi)$, 
 the centralizer of $\phi$ in $\Out(F_n),$ is
abelian.  The proof now concludes as above.
\end{proof}

Another interesting consequence of Theorem~\ref{main thm}  is the
following result.  

\begin{prop}\label{prop: solv}
Suppose $M$ is a compact surface of genus $0$ and  $G$ 
is a solvable subgroup of $\Symp^\infty_\mu(M)$
with full support of 
finite type, then $G$ is virtually abelian.
\end{prop}

The {\em Tits alternative} (see J. Tits \cite{Tits})
is satisfied by a group $G$ if every
subgroup (or by some definitions every finitely generated subgroup) 
of $G$ is either virtually solvable or contains a
non-abelian free group. This is a deep property known for
finitely generated linear groups, mapping class groups  \cite{Iv1}, \cite{mccarthy:tits}, and the outer automorphism group of a free group   \cite{bfh:tits1} \cite{bfh:tits2}. It is an important open question
for $\Diff^\omega(S^1)$ (see \cite{Farb-Shalen}). It is known that 
this property is not satisfied for $\Diff^\infty(S^1)$ (again
see \cite{Farb-Shalen}).

This naturally raises the question for surfaces.

\begin{conj}[{\bf Tits alternative}]\label{tits} 
If $M$ is a compact surface
then every finitely generated subgroup of $\Symp_\mu^\omega(M)$
is either virtually solvable or contains a non-abelian free group.
\end{conj}

In the final section of this paper we are able to use the techniques developed
to prove an important special case of this conjecture.

\begin{thm}\label{tits special case}
Suppose that $M$ is a compact oriented genus zero surface, that $G$ is a subgroup of $\Symp^\omega_\mu(M)$ and that $G$ contains
an infinite order element with entropy zero and at least three
periodic points. Then either $G$ contains a subgroup isomorphic to
$F_2$, the free group on two generators, or $G$ has an abelian subgroup 
of finite index.
\end{thm}

We observe that one cannot expect the virtually abelian (as opposed
to solvable) conclusion to hold more generally as the subgroup
of $\Symp^\omega_\mu(\T^2)$ generated by an ergodic translation
and the automorphism with matrix
\[
\begin{pmatrix}
1 &  1\\
0 &  1\\
\end{pmatrix}
\]
is solvable, but not virtually abelian.

\section{Maximal Annuli for elements of 
{$\Symp^\infty_\mu(M).$}}

\begin{defn}\label{defn: annular comp}
 Suppose $M$ is a compact oriented surface. 
 The {\em annular compactification} of an open annulus $U \subset M$
is obtained by blowup on an end whose frontier is a single point
and by the prime end compactification otherwise.  We will denote
it by $U_c$ (see 2.7 \cite{FH-ent0} for details). 
\end{defn}

 For any diffeomorphism $h$ of $M$ which 
leaves $U$ invariant there is a canonical extension to a
homeomorphism $h_c: U_c \to U_c$ which is functorial in the
sense that $(hg)_c = h_c g_c$.

\begin{defn}
Suppose $M$ is a compact genus zero surface,
$f \in \Symp^\infty_\mu(M)$,  and that the number
of periodic points of $f$ is greater than the Euler
characteristic of $M$.
 If $f$ has infinite order and entropy $0$,
we will call it a {\em multi-rotational
diffeomorphism.}  This set of diffeomorphisms 
will be denoted $\zz(M)$.
\end{defn}

 The rationale for the terminology {\em multi-rotational}
comes from the 
following result which is a distillation of several results
from \cite{FH-ent0} (see Theorems 1.2, 1.4 and 1.5 from that paper).
In particular, if $f \in \zz(M)$ then
every point of $\Int(M) \setminus  \Fix(f)$ has a well defined rotation number
(with respect to any pair of components from $\partial M \cup \Fix(f)$ and
there are non-trivial intervals of rotation numbers.

\begin{thm}\label{max-annuli}
 Suppose $M$ is a compact genus zero surface and
$f \in \zz(M)$.
The collection $\cA = \cA_f$ of maximal $f$-invariant open annuli in 
$\Int(M) \setminus \Fix(f)$ satisfies the following properties:
\begin{enumerate}
\item The elements of $\cA$ are pairwise disjoint.
\item The union $\displaystyle{\bigcup_{U\in \cA}U}$ is a full measure
open subset of {$M \setminus \Fix(f)$.}
\item \label{item:U is essential} Each $U \in \cA$
is essential in $\Int(M) \setminus \Fix(f).$

\item \label{item:rho non-constant} 
For each $U \in \cA$, the rotation number $\rho_f: U_c \to S^1$ is continuous and
non-constant.  Each component of the level set of $\rho_f$
which is disjoint from $\partial U_c$ is essential in $U$,
i.e. separates $U_c$ into two components, each containing a component of  $\partial U_c$.
\end{enumerate}
\end{thm}

We will make repeated use of the properties in the following straightforward
corollary.

\begin{cor}\label{Z centralizer}
Suppose $M$ is a compact  genus zero surface,
$f \in \zz(M)$
and $\Cent^\infty_\mu(f)$ denotes its centralizer in $\Symp^\infty_\mu(M)$. 
\begin{enumerate}
\item \label{item: cA periodic}
If $g \in \Cent^\infty_\mu(f)$ and $U\in \cA_f$ then $g(U) \in \cA_f$ and the
$\Cent^\infty_\mu(f)$-orbit of $U$ is finite.  
\item \label{item: inv level} If $U \in \cA_f$ then any $g \in \Cent^\infty_\mu(f)$ which satisfies 
$g(U) = U$ must preserve the components of each level set of $\rho_f: U_c \to S^1$.
\item If $f$ has support of finite type then the 
set $\cA_f$ of maximal annuli for $f$ is finite.
\end{enumerate}
\end{cor}
\begin{proof}
If $g \in \Cent^\infty_\mu(f)$ then $g(\Fix(f)) = \Fix(f)$
so $\Int(M) \setminus \Fix(f)$ is $g$-invariant.  Also
$f(g(U)) = g(f(U)) = g(U)$ so $g(U)$ is $f$-invariant.
Clearly $U$ is a maximal $f$-invariant annulus in 
$\Int(M) \setminus \Fix(f)$ if and only if $g(U)$ is.
This proves $g(U) \in \cA_f.$  The $\Cent^\infty_\mu(f)$-orbit of $U$
consists of pairwise disjoint annuli each with the same
positive measure.  There can only be finitely many of them
in $M.$  This proves (1).

To show (2) we observe that $g$ is a topological conjugacy from $f$ to
itself and rotation number is a conjugacy invariant.  Hence $g$ must
permute the components of the level sets of $\rho_f$.  Clearly those
level sets which contain a component of $\partial U_c$ must be
preserved. Since any other such component separates $U$, if it were
not $g$-invariant the fact that $g$ is area preserving would be
contradicted.

Since the elements of $\cA_f$ are maximal $f$-invariant annuli in
$\Int(M) \setminus \Fix(f)$ , they are mutually non-parallel in
$\Int(M) \setminus \Fix(f)$ .  Let $E$ be the union of a set of core
curves for some given finite subset of $\cA_f$.
Theorem~\ref{max-annuli}-\pref{item:U is essential} implies that each
component of {$M \setminus E$} either contains a component of $\Fix(f)
\cup \partial M$ or has negative Euler characteristic. 
The dual graph for $E$ is a tree that has one edge for each 
element of $E$ and has at
most as many valence one and valence two vertices as 
there are components of $\Fix(f)
\cup \partial M$.  Since the Euler characteristic of the tree
is $1$ there is a uniform bound on the number of its edges.
It follows that the cardinality of $E$, and hence
the number of elements of $\cA_f$, is uniformly bounded.  This proves
(3).
\end{proof}

\section{The positive entropy case.} \label{sec:positive entropy}

\begin{lemma}\label{exp factor}
Suppose that $M$ is a closed surface, that 
$f\in \Diff^2(M)$ and that $g$ commutes with $f$.
Suppose further that
$f$ has a hyperbolic fixed point $p$ of saddle type,
and that $g$ fixes $p$ and 
preserves the branches of $W^s(p,f)$. Then there is a $C^1$
coordinate function $t$ on $W^s(p,f)$ and a 
unique number $\alpha>0$ such that in 
these coordinates $g(t) = \alpha t.$  In particular $\alpha$ is
an eigenvalue of $Dg_p.$
\end{lemma}

\begin{proof}
By Sternberg linearization (see Theorem 2 of \cite{Sternberg})
there is a $C^1$ coordinate function $t(x)$ on
$W^s(p,f)$ in which the restriction of $f$ (also denoted $f$)
satisfies $f(t) = \lambda t$ where $\lambda \in (0,1)$ is
the eigenvalue of $Df_p$ corresponding to the eigenvector
tangent to $W^s(p,f)$.

In these coordinates $g(t) = \lambda^{-n} g( \lambda^n t).$ Applying
the chain rule gives $g'(t) = g'( \lambda^n t).$ Letting $n$ tend to
infinity we get $g'(t) = g'(0)$ so $g'(x)$ is constant and $g(t) =
\alpha t$ where $\alpha = g'(0).$
\end{proof}

\begin{defn} \label{preserve branches}
Suppose that $M$ is a compact surface and that $f \in \Diff^r(M)$  has a
hyperbolic fixed point $p$ of saddle type.   We define $\Cent^r_p(f)$ to be
the subgroup of elements of $\Diff^r(M)$ which commute with $f$, fix
the point $p,$ and preserve the branches of $W^s(f,p).$ Let $\R^+$
denote the multiplicative group of positive elements of $\R$.  The
{\em expansion factor homomorphism}
\[
\phi: \Cent^r_p(f) \to \R^+,
\] 
is defined by $\phi(g) = \alpha$ where $\alpha$ is the unique number
given by Lemma~(\ref{exp factor}) for which $g(x) = \alpha x$ in
$C^1$ coordinates on $W^s(f,p).$  It is immediate that $\phi$ is
actually a homomorphism.  We also observe that $\phi(g)$ is
just the eigenvalue of $Dg_p$ whose eigenvector is tangent to 
 $W^s(f,p).$
\end{defn}

The following result is due to Katok \cite{Katok2} who stated
it only in the analytic case, but gave a proof very similar to the one
below.
\bigskip

\noindent {\bf Proposition~\ref{prop cyclic}} ([Katok \cite{Katok2}]). {\em
Suppose $G$ is a subgroup of $\Diff^2(M)$ which has full support of
finite type and $f \in \Diff^2(M)$ has positive
topological entropy.  Then the centralizer of $f$
in $G$, \  $\Cent^2(f,G)$, is virtually cyclic.  Moreover, every infinite
order element of $\Cent^2(f,G)$ has positive topological entropy.}

\begin{proof}  
By a result of Katok \cite{katok:horseshoe} there is a hyperbolic periodic
point $p$ for $f$ of saddle type with a transverse homoclinic point $q$.  
Let  $\phi: \Cent^2_p(f) \to \R^+$ be the expansion factor homomorphism of  Definition~\ref{preserve branches} and  let   $ H =\Cent^2_p(f) \cap G $.    Then $H$ has finite index in $\Cent^2(f,G)$
because otherwise the {$\Cent^2(f,G)$} orbit of $p$ is 
infinite and consists
of hyperbolic fixed points of $f$ all with the same eigenvalues.
This is impossible because a limit point would be a non-isolated
hyperbolic fixed point.  For the main statement of the proposition, it suffices to show that $H$ is cyclic and for this it  suffices to  show that the restriction    $\phi|_H: H \to \R^+$  is injective and has a discrete image.

To show that $\phi|_H$ is injective it suffices  to show that
  if $g \in H$ and $\phi(g) = 1$ then 
$g = id.$   But $\phi(g) = 1$ implies $W^s(p,f) \subset \Fix(g)$.
Note that if $\phi': \Cent^2_p(f^{-1}) \to \R^+$  is the expansion homomorphism for $f^{-1}$ then $\phi(g)\phi'(g) = det(Dg_p) = 1$ so we also
know that $W^u(p,f) \subset \Fix(g)$. Hence we need only show that 
this implies $g = id$.  Let $J_s$ be the interval in
$W^s(p,f)$ joining $p$ to $q$ and define $J_u \subset W^u(p,f)$
analogously.  Define $K_n = f^{-n}(J_s) \cup f^{n}(J_u)$.  Then the
number of components of the complement of $K_n$ tends to $+\infty$
with $n.$  Since $g$ has full support of finite type,  each component $V$ of the complement of $K_n$ contains a component of $M \setminus \Fix(g)$, in contradiction to the fact that $M \setminus \Fix(g)$ has only finitely many components.   We conclude that $\phi$ is injective.

To show that $\phi|_H$ has discrete image we assume that
there is a sequence $\{g_n\}$ of elements of $H$ such that
$\lim \phi(g_n) = 1$, and show that $\phi(g_n) = 1$ for $n$
sufficiently large.  
 Let $I_s = f^{-1}(J_s)$ and $I_u = f(J_u)$. Since $I_s$ and
$I_u$ are compact  intervals and the point $q$ is a point in the interior
of each where they intersect transversely, there is a neighborhood
$U$ of $q$ such that $U \cap I_s \cap I_u =\{q\}.$ But for $n$
sufficiently large $g_n(q) \in U$ and $g_n(q) \in (I_s \cap I_u).$ It
follows that $g_n(q) =q$ for $n$ sufficiently large
 and hence $\phi(g_n) = 1.$  This proves
that the image of $\phi_H$ is discrete and hence that $H$ is cyclic.

 Each non-trivial element $h \in H$ has $\phi(h) \ne 1$ 
and $\phi'(h) \ne 1$.  Hence $p$ is a hyperbolic fixed point of $h$
and $q$ is a transverse homoclinic point for $h$.  It follows that
$h$ has positive entropy {(see Theorem 5.5 of \cite{smale:DDS}).}
\end{proof}

\begin{prop}\label{A +entropy}
Suppose $G$ is a  subgroup of
$\Diff^2(M^2)$ which has full support of finite type and $A$ is an
abelian normal subgroup of $G$.  If there is an element $f \in A$ with
positive topological entropy then $G$ is virtually cyclic.
\end{prop}

\begin{proof}
It follows from Proposition~\ref{prop cyclic} that the group $A$ is virtually cyclic.  Since $f \in $ has positive entropy it is infinite order and generates a finite index subgroup $A_f$ of $A$. Hence there exists a positive integer $k$ such that $a^k \in A_f$ for all $a \in A$.  In particular, for each $g \in G$,    we have $g f^k  g^{-1} = (g f g^{-1})^k  = f^m$ for some $m \in \Z$ .   Since $f$ and $g f g^{-1}$ are conjugate,
the topological entropy $\ent(f) = \ent(g f g^{-1})$. Hence
\[
\ent(g f g^{-1})^{k} = |k| \ent(g f g^{-1})  = |k| \ent(f),
\text{ and  }\ent(f^{m}) = |m| \ent(f).
\]
We conclude that $m = \pm k$ and hence that $g f^k g^{-1} = f^{\pm k}$ for all $g \in G$.  Let $G_0$ be the subgroup of index at most two of $G$ such that $g f^kg^{-1} = f^k$ for all $g \in G_0$.  Then $G_0 \subset \Cent^2(f^k,G)$ and Proposition~\ref{prop cyclic} completes the proof. 
\end{proof}

\section{Mean Rotation Numbers}

In this section we record some facts (mostly well known)
 about rotation numbers for homeomorphisms of area preserving
homeomorphisms of the closed annulus.  In subsequent sections
we will want to apply these results when $M$ is a surface,
$f \in \Symp^\infty_\mu(M)$, and $U$ is an
$f$-invariant {\em open} annulus.  We will do this by considering
the extension of $f$ to the closed annulus $U_c$ which is the annular
compactification of $U$.  When the annulus $U$ is understood,
we will use $\rho_f(x)$ to mean the rotation number of $x \in U$ with
respect to the homeomorphism $f_c: U_c \to U_c$ of the closed
annulus $U_c$.

\begin{defn} Suppose that $f:\A \to \A$ is a homeomorphism of a closed
annulus $\A = S^1 \times [0,1]$ preserving a measure $\mu$.   For each lift $\ti f$, let $\Delta_{\ti f}( x) = p_1(\ti f(\ti x)) - p_1(\ti x)$ where $\ti x$ is a lift of $x$ and $p_1: \ti A   = \ti R \times [0,1] \to \R$ is projection onto the first factor.  
As reflected in the notation, $\Delta_{\ti f}(x)$ is independent
of the choice of lift $\ti x$ and hence may be considered
as being defined on $\A.$

If $X \subset \A$ is an $f$-invariant $\mu$-measurable
set then the {\em mean translation number relative to} $X$ 
 and $\ti f$ is defined to be
\[
\cT_{\mu}(\ti f, X) = \int_X \Delta_{\ti f}(x) \ d\mu.
\]
We define the {\em mean rotation number relative to} $X$, $\rho_\mu(f,X)$
to be the coset of $\cT_{\mu}(\ti f, X)$ mod $\Z$ thought of as an element of $\T = \R/\Z.$
\end{defn}

The mean rotation number is independent of the 
choice of lift $\ti f$ since different
lifts give values of $\cT(\ti f, X)$ differing by an element of $\Z$.

 We define the {\em translation number} of $x$ with
respect to $\ti f$ (see e.g. Definition~2.1 of \cite{FH-ent0}), by 
\[
\tau_{\ti f}(x) = \lim_{n \to \infty} \frac{1}{n}\Delta_{\ti f^n}(x).
\]
A straightforward application of 
the Birkhoff ergodic theorem implies that $\tau_{\ti f}(x)$ is
well defined for almost all $x$ and 
\[
\cT_\mu(\ti f, X) = \int_{X} \tau_{\ti f}(x)  \ d\mu.
\]

If $X$ is also $g$-invariant and $g$ preserves $\mu$ then
\[
\Delta_{\ti f \ti g}(x) = \Delta_{\ti g}(x) + \Delta_{\ti f}(\ti g(x))
\]
and hence integrating we obtain
\[
\cT_\mu(\ti f \ti g, X) = \cT_\mu(\ti f,X) + \cT_\mu(\ti g,X) \text{ and }
\rho_\mu(fg, X) = \rho_\mu(f,X) + \rho_\mu(g,X),
\]
\i.e. $f \mapsto \rho_\mu( f, X)$ is a homomorphism from the 
group of $\mu$ preserving homeomorphisms of $\A$ to
$\T = \R/\Z.$ 
Hence if $h = [g_1,g_2]$
for some $g_i : \A \to \A$ that preserve $\mu$ then $\rho_\mu(h) = 0.$

 \begin{lemma}\label{int rho}
Let $f: \A \to \A$ be an area preserving homeomorphism of a closed 
annulus isotopic to the
identity.  Suppose $X \subset \A$ is an $f$-invariant subset with
Lebesgue measure $\mu(X) > 0.$  If $\rho_\mu(f,X) = 0$
then $f$ has a fixed point in the interior of $\A$.
\end{lemma}

\begin{proof}
Replacing $X$ with $X \cap \Int(\A)$ we may assume $X \subset \Int(\A).$
If $\tau_{\ti f}(x) = 0$ on a positive measure subset of $X$ then the
fixed point exists by Proposition~2.4 of \cite{FH-ent0}.  
Otherwise for some $\epsilon >0$ 
there is a positive measure $f$-invariant subset $X^+$ on which $\rho_f > \epsilon$ and
another $X^-$ on which $\rho_f < -\epsilon$.  If $x \in X^+$ is recurrent
and not fixed then it is contained in a positively recurring disk 
which is disjoint from its $f$-image.
Similarly if $y \in X^-$ is recurrent
and not fixed then it is contained in a  negatively recurring disk
which is disjoint from its $f$-image.
Theorem 2.1 of \cite{F-Poincare}  applied to $f$ on the open annulus $\Int(\A)$ 
then implies there is a fixed point in $\Int(\A).$ 
\end{proof}

\begin{lemma}\label{lem: periodic}
 Suppose that $f: \A \to \A$ is an area preserving homeomorphism 
of a closed annulus which is
is isotopic to the identity.  Suppose also that $V \subset \A$ is
a connected open set which is inessential in $\A$ and such that 
$f^m(V) = V$ for some $m \ne 0$. Then there is a full measure subset $W$
of $V$ on which $\rho_f$ assumes a constant rational value.
\end{lemma}

\begin{proof}
Replacing $V$ with $V \cap \Int(\A)$ we may assume $V \subset
\Int(\A)$.   Replacing $f$ with $f^m$ it will suffice
to show that $f(V) = V$ implies $\rho_f =0$ on a full measure subset
of $V$. Since $V$ does not contain a simple closed curve that
separates the components of $\partial \A$, both components of
$\partial \A$ belong to the same component $X$ of $\A \setminus V$ by
Lemma 3.1 of \cite{FH-ent0}.   Note that $D := \A \setminus X$ 
contains $V$, is contained in $\Int(\A)$ and is
$f$-invariant.  Lemma 3.2 of \cite{FH-ent0} implies that $D$ is
connected and it is simply connected since its complement $X$  is connected.
Hence $D$ is an open disk.

By the Brouwer plane translation theorem $D$ contains a point of
$\Fix(f)$. Let $\ti \A$ be the universal cover of $\A$ and let $\ti D
\subset \ti \A$ be an open disk which is a lift of $D.$ Since $f$ has a
fixed point in $D$ there is a lift $\ti f: \ti \A \to \ti \A$ which has
a fixed point in $\ti D.$ Therefore $\ti f(\ti D) = \ti D.$ If
$p_1(\ti D)$ is bounded then it is obvious that $\rho_f$ is constant
and $0$ on $D$.  Otherwise, note that by Poincar\'e recurrence
and the Birkhoff ergodic theorem there is a full measure subset $\ti
W$ of $\ti D$ consisting of points which are recurrent and have a well
defined translation number.  Calculating the translation number of a
point $\ti x \in \ti W$ on a subsequence of iterates which converges
to a point of $\ti \A$ shows that it must be $0$. Hence $W$, the
projection of $\ti W$ to $D$, has the desired properties.
\end{proof}

\section{Pseudo-rotation subgroups of $\Symp^r_\mu(M)$.}

\begin{defn} Suppose $M$ is a compact oriented surface.
  A {\em pseudo-rotation subgroup} of $\Symp^r_\mu(M)$ with $r
  \ge 1,$ is a subgroup $G$ with the property that every non-trivial
  element of $G$ has exactly $\X(M)$ fixed points. An element $f \in
  \Symp^r_\mu(M)$ will be called a {\em pseudo-rotation} provided the
  cyclic group it generates is a pseudo-rotation group.
\end{defn}

Our definition may be slightly non-standard in that we consider
finite order elements of $\Symp^r_\mu(M)$ to be pseudo-rotations.
We observe that if $\X(M) < 0$ then
any pseudo-rotation group is trivial.  Since we assume $M$ is oriented
we have only the cases that $M$ is $\A,\ \D^2,$ or $S^2$ for which any
pseudo-rotations must have precisely $0,\ 1,$ or $2$ fixed points
respectively.  By far the most interesting case is $M = S^2$, since
we will show in Lemma~\ref{A-D abelian} that when $M = \D^2$ or $\A^2$ 
any pseudo-rotation group is abelian. This is not 
the case when $M = S^2$ since $SO(3)$ acting
on the unit sphere in $\R^3$ is a pseudo-rotation group.

There are three immediate  but
very useful facts about pseudo-rotations which we summarize
in the following Lemma:

\begin{lemma} \label{useful facts}
Suppose $M$ is a compact surface and $f \in \Symp^r_\mu(M)$
is a pseudo-rotation.
\begin{enumerate}
\item Either $f$ has finite order
or $\Fix(f) = \Per(f).$
\item If $\Fix(f)$ contains more than $\X(M)$ points then
$f = id.$  In particular if $f,g$ are elements of
a pseudo-rotation group which agree at more than $\X(M)$ points
then $fg^{-1} =id$ so $f = g.$
\item If $M = \D^2$, then the one fixed point of 
$f$ is in the interior of $\D^2$. 
\end{enumerate}
\end{lemma}

\begin{proof}

Parts (1) and (2) are immediate from the definition of pseudo-rotation
and part (3) holds because otherwise the restriction of $f$ 
to the interior would have
recurrent points but no fixed point, contradicting the Brouwer
plane translation theorem.
\end{proof}

We will make repeated use of these properties.

\begin{prop}\label{rho constant}
Suppose $f \in \Symp^r_\mu(M)$ is a pseudo-rotation.
Then $U = \Int(M) \setminus \Fix(f)$ is an open annulus
and $\rho_f: U_c \to \T = \R/\Z$ 
is a constant function. Moreover, 
{if $G \subset \Symp^r_\mu(M)$ is an abelian pseudo-rotation group,
with $\Fix(g) = \Fix(f)$ for all $g \in G$, then the assignment 
$g \mapsto \rho_g$,  for $g \in G$, is an injective homomorphism,
so, in particular,
if $g$ has infinite order then the constant $\rho_g$ is irrational.}
\end{prop}

\begin{proof} The only possibilities for $M$ are $S^2, \A$ or
 $\D^2$.   If $M = \D^2$, as remarked above, its one fixed 
point must be in the interior of $M$.
Hence in all cases $U$   is an open annulus.
If $\rho_f$ is not constant on $U_c$ then the 
Poincar\'e-Birkhoff Theorem (see Theorem (2.2) of
\cite{FH-ent0} for example) implies  there are 
periodic points with arbitrarily large period.
We may therefore conclude that $\rho_f$ is
constant and equal to $\rho_\mu(f)$.  
{ Suppose now that $G$ is an abelian
pseudo-rotation subgroup containing $f$ and $\Fix(f) = \Fix(g)$ for
all $g \in G$, so $U$ is $G$-invariant.
Since $\rho_\mu: G \to \T$ is a homomorphism
$\rho_\mu(g)$ being rational implies
$\rho_\mu(g^n) = 0$ for some $n$. Then Lemma~\ref{int rho}
implies the existence of a fixed point in $U$ for $g^{n}$.  Hence 
$g^{n} = id$. We conclude that if $g$ has infinite order then
 $\rho_\mu(g)$ is irrational.  If  $\rho_\mu(f) =  \rho_\mu(g)$
then  $\rho_\mu(fg^{-1}) = 0$ and the same argument shows 
$fg^{-1} = id$, so the assignment $f \mapsto \rho_\mu(f)$ 
for $f$ in the abelian group $G$ is injective.}
\end{proof}

We observed above that if $f$ is a non-trivial pseudo-rotation then
either $f$ has finite order or 
$\Per(f) = \Fix(f).$  We now prove the converse to this statement.

\begin{prop}\label{prop: pseudo-r}
Suppose $M$ is $\A,\ \D^2$ or $S^2$ and 
$G$ is a subgroup of $\Symp^\infty_\mu(M)$ with the property that every 
non-trivial element $g$ of $G$ either has finite order or 
satisfies $\Fix(g) = \Per(g)$.  Then $G$ is a pseudo-rotation group. 
\end{prop}

\begin{proof} 
{No element  $g \in G$ can have positive entropy since that
would imply $g$ has points of arbitrarily high period
(see Katok \cite{katok:horseshoe}), a contradiction.}
 If $g \in G$, then it must have at least $\X(M)$
fixed points because the Lefschetz index of a fixed point of $g$ is at
most $1$ (see \cite{simon:index}, for example). 
 If $g$ is non-trivial and 
has more than $\X(M)$ fixed points it follows from 
part (\ref{item:rho non-constant}) of Theorem~\ref{max-annuli}
that there is an $g$-invariant annulus $U \subset int(M)$ for which
the rotation number function in continuous and non-constant.
It then follows from Corollary~2.4 of \cite{franks:recurrence} 
that there are periodic  points with arbitrarily high period,
contradicting our hypothesis. We conclude
that either $g$ has finite order or $\card(\Per(g)) = \X(M).$ 
\end{proof}

\begin{lemma} \label{three components} Suppose $U \subset M$ is an
open annulus and that $C_1$ and $C_2$ are disjoint closed connected
sets in $U$  each of which separates the ends of $U$ and 
such that both  $U \setminus C_1$ and $U \setminus C_2$
have two components.  Then $M \setminus (C_1 \cup C_2)$ has three
components: one with frontier contained in $C_1$, one with frontier
contained in $C_2$ and an open annulus with frontier contained in
$(C_1 \cup C_2)$.
\end{lemma}

\begin{proof}    Clearly $M \setminus C_i$ has two components.
Let $X_1$ and $X_2$ [resp. $Y_1$ and $Y_2$] be the  components of $M \setminus C_1$ [resp. $M \setminus C_2$] labeled so that $C_2 \subset X_2$ and $C_1 \subset Y_1$.    Then $X_1$ and $Y_2$ are components of  $M \setminus (C_1 \cup C_2)$  with frontiers contained in $C_1$ and $C_2$ respectively and it suffices to show that $V = X_2 \cap Y_1$ is an open annulus.    
We use the fact
that an open connected subset of $S^2$ whose complement has two components is
an annulus. {Of course the same is true for $M$ which may be considered
as a subsurface of $S^2$.} But 
$X_1 \cup C_1 = \cl(X_1) \cup C_1$ is connected and similarly so is
$Y_2 \cup C_2 = \cl(Y_2) \cup C_2$.  Hence $M \setminus V $ has two 
components, namely $X_1 \cup C_1$ and $Y_2 \cup C_2$.  So $V$ is an annulus.
\end{proof}

Our next result includes a very special case of Proposition~(\ref{prop:
  cent}) namely we show that the centralizer of an infinite order
pseudo-rotation is virtually abelian.  In fact, for later use, we need
to consider a slightly more general setting, namely, not just the
centralizer of a single pseudo-rotation, but the centralizer of an
abelian pseudo-rotation group all of whose non-trivial elements have
the same fixed point set.  

\begin{lemma}\label{2periodic}
Suppose $A$ is an abelian pseudo-rotation subgroup of $\Symp^\infty_\mu(M)$ 
containing elements of arbitrarily large order. Suppose
also that there is a set $F$ with $\X(M)$ points such that $\Fix(f) = F$ for all
non-trivial $f \in A.$ 
Then the subgroup $\C_0$ of the centralizer, $\Cent^\infty_\mu(A),$ of $A$
 in $\Symp^\infty_\mu(M)$
consisting of those elements which pointwise fix $F$, is abelian and
has index at most $2$ in $\Cent^\infty_\mu(A).$
\end{lemma}

\begin{proof} 
Elements of $A$ must have entropy $0$ since positive entropy implies
the existence of infinitely many periodic points by a result of Katok,
\cite{katok:horseshoe}.  Let $U$ be the open annulus $M \setminus (F
\cup \partial M)$.  The group  $\Cent^\infty_\mu(A)$ preserves 
$F \cup \partial M$ and hence $U$. The
subgroup $\C_0 \subset \Cent^\infty_\mu(A)$, whose elements fix the ends of $U$,
has index at most $2$.

We claim that the elements of $\C_0$ all have entropy $0$. Clearly we
need only consider elements of infinite order since all finite order
homeomorphisms have entropy $0$.  Hence the claim follows from
Corollary~(\ref{dichotomy}) if  there is an element of infinite
order in $A$.  Otherwise if $ \C_0$ contains an element with
positive entropy then it contains an element $g$ with a hyperbolic
fixed point $p \in U$.  Each point in the $A$-orbit of $p$ is in
$\Fix(g)$ and has the same set of eigenvalues for the derivative $Dg$.
It follows that the $A$-orbit of $p$ has finite cardinality, say $m$,
and hence that $p \in \Fix(f^m)$ for all $f \in A$.  
Part (2) of Lemma~\ref{useful facts} then implies that 
$f^m = id$ for all $f \in A$, {if $m > 2$.}  This contradicts
the fact that $A$ contains elements of arbitrarily high order.
This completes the proof that all elements of $\C_0$ have entropy zero.

To prove that $\C_0$ is abelian, we will show that each commutator $h$
of two elements in $\C_0$ is the identity by assuming that $h$ is non-trivial and arguing to a contradiction.  Since $h$ is a commutator
the map $h_c: U_c \to U_c$ has mean rotation number $\rho_\mu(h_c) =
0$ and hence, by Lemma~\ref{int rho}, $h$ has a fixed point in $U$.
Therefore $\Fix(h)$ contains more than $\X(M)$ points.  If $h$ has finite order  then in a suitable averaged metric it is
an isometry of $M$.  But then each fixed point must
have Lefschetz number $+1$ and hence by the Lefschetz theorem $h$
has $\X(M)$ fixed points, a contradiction.
We conclude therefore that $h$ has infinite order
and hence satisfies the hypothesis 
of Theorem~(\ref{max-annuli}). By this theorem there exists an
element $V \in \cA_h$ and it must be the case that $V \subset U$
since $V$ cannot contain a point of $\Fix(h) \supset F$.  We will
show this leads to a contradiction, either by showing that $h$ has a
fixed point in $V$ or by showing that some non-trivial $f \in A$ has
a fixed point in $U$. We may then conclude $h = id$ and hence that
$\C_0$ is abelian.

Suppose first that $V$ is
inessential in $U$.   Elements of $A$ commute with $h$ and hence
permute the elements of $\cA_h$ by Corollary~\ref{Z centralizer}. 
Since there can be only finitely many
elements of $\cA_h$ of any fixed area, the $A$-orbit
of $V$ is finite.   Hence there 
is $m$ such that $f^m(V) = V$ for all $f \in A$. 
The union of  $V$  with the component of its
complement which is disjoint from $F$ is an open disk $D \subset
U$ satisfying $f^m(D) = D$ for all $f \in A.$ 
It follows from the Brouwer plane translation theorem that $f^m$ 
has a fixed point in $D.$ If the order of $f$ is greater than $m$ then
$f^m$ is a non-trivial element of $A$ with fixed points not in $F$
contradicting the assumption that $\Fix(f) = F$ for all non-trivial $f \in A$.
So the possibility that $V$ is inessential in $U$ is contradicted.

We are now reduced to the case that $V$ is essential in $U$.  
It follows from part (\ref{item: cA periodic}) of 
Corollary~\ref{Z centralizer} that each element of $A$ preserves
$V$ or maps it to a disjoint element of $\cA_h$.  
Since the annulus $U$ is $A$-invariant and $V$ is essential in 
$U$ it follows that $V$ is $A$-invariant as the alternative
would contradict the fact that $A$ preserves area.
We want to replace $V$ with a slightly smaller essential
$V_0 \subset V$ which has the property that its frontier (in $M$) lies
in $V$ and has measure $0$. To do this we observe that $\rho_h$ is non-constant on $V$  and hence
has uncountably many level sets.  Hence there must
be two of its level sets $C_1, C_2$ which have measure $0$. Let $V_0$ be the
essential open annulus in $V$ whose frontier lies in $C_1 \cup C_2$ (see
Lemma~\ref{three components}).  Then $\mu(V_0) = \mu(\cl_{U_c}(V_0))$ since
$\cl_{U_c}(V_0) \setminus V_0 \subset C_1 \cup C_2.$  It follows from 
part (\ref{item: inv level}) of Corollary~\ref{Z centralizer} that
each $C_i$ is $A$-invariant and hence $V_0$ is also.

As a first subcase, suppose that $ \cl_{U_c}(V_0)$ is $g$-invariant 
for each $g \in
\C_0$.  Then $h$ is a commutator of elements 
that preserve $\cl_{U_c}(V_0)$, so
\[
\rho_\mu(h, \cl_{U_c}(V_0)) =  0,
\] (measured in $U_c$).  Since $V_0$ differs
from $\cl_{U_c}(V_0)$ in a set of measure zero, $\rho_\mu(h, V_0) = 0$ also.
The translation number  with respect to a lift $\ti h$ of a point $x \in V_0$
can be measured in either $U_c$ or $(V_0)_c$ giving $\tau_{\ti h}(x)$
and $\tau_{\ti h|_{\ti V_0}}(x)$.  But $x \in V_0$ lies on a compact
$h$-invariant level set $C_0$ which is in the interior of both $U_c$ 
and $(V_0)_c$.  This implies $\tau_{\ti h}(x) = \tau_{\ti h|_{\ti V_0}}(x)$.
Hence we conclude that $\rho_\mu(h) =  0,$ measured in $(V_0)_c$ and
by Lemma~\ref{int rho}, $h$ has a fixed point in $V_0$, a
contradiction.
  
The last subcase is that there exists $g \in \C_0$ such
that $g(V_0) \not \subset \cl_{M}(V_0) = \cl_{U_c}(V_0)$.  
Choose a component $W$ of $g(V_0) \cap
(U \setminus \cl_{U_c}(V_0))$.  Since $g$ leaves $U$ invariant and preserves area and since $g^{-1}(W) \cap W = \emptyset$, $W$ cannot contain a closed  curve
$\alpha$ that is essential in $U$.  Thus   $W$ is inessential in $U$.
As we noted above $V_0$ is $A$-invariant.  
Since  $g \in \C_0$, it commutes with $A$  so $g(V_0)$ is also
$A$-invariant. It follows that $A$ permutes the components of $g(V_0)
\cap (U \setminus \cl_{U_c}( V_0))$. In particular for some $m >
0$ we have $f^m(W) =W$ for all $f \in A$.  Letting $D$ be the union of
$W$ with any components of its complement which do not contain 
 ends of $U$ we
conclude that $D$ is an open disk and $f^m(D) = D$ for all $f \in A$.  By the
Brouwer plane translation theorem there is a point of $\Fix(f^m)$ in $D$.
Since $f^m$ is a pseudo-rotation and has more than $\X(M)$ fixed points,
$f^m = id$.  Since this holds for any $f \in A$ we have
contradicted the hypothesis that $A$ contains elements of arbitrarily high
order.
\end{proof}

\begin{lemma}\label{A-D abelian}
Suppose that $G$ is a pseudo-rotation subgroup
of $\Symp^r_\mu(M)$ where $M = \A$ or $\D^2$ and $r \ge 2$.
Then $G$ is abelian and $\Fix(g)$ is the
same for all non-trivial $g \in G.$
\end{lemma}

\begin{proof} 
If $M =\A$ then $\Fix(g) = \emptyset$ for all non-trivial $ g \in G$. Since $\rho_\mu([f,g]) = 0$  for each $f,g \in G$,     Lemma~\ref{int rho} implies that each
$[f,g]$ has a fixed point in the interior of 
$\A$ and hence is the identity.   Thus each $f$ and $g$ commute  and we are done.

We assume for the remainder of the proof that 
 $M = \D^2$.  For each $f \in G$  
we consider $\hat f$, the restriction of $f$ to $\partial \D^2.$   Let $\hat G$ be the image in $\Diff^r(S^1)$ of $G$ under the homomorphism $f \mapsto \hat f$.

Lemma~\ref{useful facts} (3) implies that $f$ fixes a point in the interior of $\D^2$.  If  $f ^k$ fixes a point in $\partial \D^2$   for some $k \ge 1$ then $f^k$ has more than $\X(M)$ fixed points and so is the identity.  We conclude that if $\hat f$
has a point of period $k$ then $f$ is periodic of period $k$.
In particular, if  $\hat f:S^1 \to S^1$ has a fixed point then $f = id.$
{This proves that $\hat G$ acts freely on $S^1$.
It also proves that the restriction homomorphism
$f \mapsto \hat f$ is an isomorphism  $G \to \hat G$.
Since $\hat G$ acts freely on $S^1$ it follows from 
H\"older's Theorem (see, e.g., Theorem 2.3 of \cite{Farb-Shalen})
that $\hat G$ (and hence $G$) is abelian.}

Suppose $f,g \in G$ and let $\Fix(f) =  \{q\}$.  Since $f$ and $g$ commute, $g$ preserves $\{q\}$ and so fixes $q$.  This proves that $\Fix(f) = \Fix(g)$
and completes the proof. 
\end{proof}

  If $f \in \Diff^r(S^2)$ and $p \in \Fix(f)$ we can compactify $S^2 \setminus \{p\}$ by blowing up
$p$, i.e. adding a boundary component on which we extend $f$ by letting it
be the projectivization of $Df_p$.  More precisely we define
$\hat f_p: S^1 \to S^1$ by considering the boundary as the unit circle in
$\R^2$ and letting
\[
\hat f_p(v) = \frac{Df_p(v)}{|Df_p(v)|}.
\]

If $\Fix(f)$ contains two points then we may blow up these points
and obtain the annular compactification $\A$ of $S^2 \setminus \Fix(f)$.

\begin{cor}\label{fp-stabilizer}
Suppose that $G$ is a pseudo-rotation subgroup
of $\Symp^\infty_\mu(S^2)$ and that $G_p$ is the stabilizer of a 
point $p \in S^2$.  Then $G_p$ is abelian and there 
exists $q \in S^2$ such that $\Fix(g) = \{p,q\}$ for  
all non-trivial $g \in G_p$.    
\end{cor}

\begin{proof}   
Suppose that $f \in G_p$.  Blowing up the point $p$ and extending $f$
we obtain $\bar f: \D^2 \to \D^2.$  This construction is functorial 
in the sense that if $h = fg$ then $\bar h = \bar f \bar g.$  Hence the
assignment $f \mapsto \bar f$ is a homomorphism from $G_p$
to $\Symp^\infty_\mu(\D^2)$.  This homomorphism is injective 
by part (2) of Lemma~\ref{useful facts}.
The result now follows from Lemma~\ref{A-D abelian}.
\end{proof}

It is not quite the case that for abelian pseudo-rotation subgroups 
 of $\Symp^\infty_\mu(S^2)$ that
the fixed point set for non-trivial elements is independent of the element.
There is essentially one exception, namely the group generated by rotations
of the sphere through angle $\pi$ about two perpendicular axes.

\begin{lemma}\label{lem: abelian rot}
If $G$ is an abelian pseudo-rotation subgroup of $\Symp^\infty_\mu(S^2)$ then
either $\Fix(g)$ is the same for every non-trivial element $g$ of $G$
or $G$ is isomorphic to $\Z/2\Z \oplus \Z/2\Z$.  In the latter case 
the fixed point sets of non-trivial elements of $G$ are pairwise
disjoint and hence their union contains six points.
\end{lemma}

\begin{proof}
Let $g,h$ be distinct non-trivial elements of $G$ and suppose 
$\Fix(g) \ne \Fix(h)$.
Since $G$ is abelian $g(\Fix(h)) = \Fix(h)$.
{Since $\Fix(g) \ne \Fix(h)$ and each contains two points, these
sets are disjoint.} Hence
$g$ switches the two points of $\Fix(h)$.  In this case $g$ has
two fixed points and a period $2$ orbit and hence $g^2$ has four fixed
points and must be the identity.  We conclude that $g$ has order two.
Switching the roles of $g$ and $h$ we observe that $h$ has order two.

Let $G_0$ be the abelian group generated by $g$ and $h$ 
so $G_0 \cong \Z/2\Z \oplus \Z/2\Z$. The possible actions of
an element of $G_0$ on $\Fix(g) \cup \Fix(h)$ are: fix all four points,
switch the points of one pair and fix the other, or switch both pairs. 
If $f \in G$ then $f$ preserves $\Fix(g)$ and  $\Fix(h)$ so its action on
$\Fix(g) \cup \Fix(h)$ must be the same as one of the elements of $G_0$.
Since $f$ agrees with an element of $G_0$ at four points it must, in fact,
be an element of $G_0$.  Hence any abelian pseudo-rotation group
containing non-trivial elements whose fixed point sets do not 
coincide must be isomorphic to $\Z/2\Z \oplus \Z/2\Z$.

Suppose now  $G \cong \Z/2\Z \oplus \Z/2\Z$ is 
a pseudo-rotation subgroup of $\Symp^\infty_\mu(S^2).$
The three non-trivial elements of $G$  must have pairwise disjoint fixed
point sets.  To see this note that otherwise there would be two distinct
elements,  say, $g$ and $h$ with a common fixed point. As observed above this
implies $\Fix(g) = \Fix(h)$. But any two non-trivial elements of $G$
generate it and hence $\Fix(f)$ is the same for all elements $f \in G.$
Then by Proposition~\ref{rho constant} there is an injective homomorphism
from $G$ to $\T$. This is a contradiction since $\T$ contains a unique
element of order two.

\end{proof}

\begin{remark}\label{rmk: 6pts}
Suppose in the previous lemma
$G \cong \Z/2\Z \oplus \Z/2\Z$.  Let
$F$  be the union of three sets $F_i = \{a_i, b_i\},\  1\le i \le 3$
where $F_i = \Fix(h_i)$ and $h_i$ is one of the non-trivial elements
of $G$. The elements of $G$ preserve each of the sets $F_i$ and if $h$ is a
non-trivial element of $G$ it must fix the points of one of the $F_i$'s
while switching the points of the other two (since  $h$ 
cannot fix four points).
\end{remark}

Suppose $A$ is a subgroup of the circle $\T = \R/\Z$  
and $\G = A \rtimes_\phi (\Z/2\Z)$ is the
semidirect product determined by the homomorphism $\phi: \Z/2\Z  \to \Aut(A)$ 
given by $\phi(1) = i \in \Aut(A)$ where 
$i: A \to A$ is the involution
 $i(h) = -h$.  Then $\G$ is called a {\em generalized dihedral group.}

\begin{lemma}\label{lem: F finite}
Suppose $G$ is a pseudo-rotation subgroup of $\Symp^\infty_\mu(S^2)$
such that either $G$ leaves invariant a non-empty finite set $F \subset S^2$, or
$G$ has a non-trivial normal solvable subgroup.
{Then either 
\begin{enumerate}
\item There is a $G$-invariant set $F_0$ containing two points in which case
$G$ is isomorphic to a subgroup of
the generalized dihedral group $A \rtimes_\phi (\Z/2\Z)$ for some subgroup
$A$ of the circle $\T,$ or
\item Every $G$-invariant set contains more than two points in which case
$G$ is finite.
\end{enumerate}
Moreover, if the set $F_0$ in (1) is pointwise fixed then $G$ is isomorphic
to a subgroup of $\T.$
}
\end{lemma}

\begin{proof}
{Suppose first $G$ has a non-trivial normal solvable subgroup $N$.
We will reduce  this case to the case that there is a finite $G$-invariant set.
Let $H$ be the last non-trivial
element of the derived series of $N$ so $H$ is an abelian subgroup.
Since $H$ is characteristic in $N$ it is normal in $G$.}

By Lemma~\ref{lem: abelian rot} the set
\[
F = \bigcup_{h \in H, h\ne id} \Fix(h)
\]
is finite.
If $g \in G$ and $h \in H$ then $ghg^{-1} \in H$.  Since $\Fix(ghg^{-1}) = g(\Fix(h))$, it follows
that $g(F) = F.$  Hence we need only prove the conclusion of our result
under the assumption that $G$ leaves invariant a finite set $F$.

 Let $A$ be the finite index subgroup of $G$ which
pointwise fixes $F$.  
{By Corollary~\ref{fp-stabilizer} $A$ is abelian.
If $F$ contains more than two points then $A$
must be trivial and $G = G/A$ is finite since $A$ had finite index.  
If $F$
is a single point then by Corollary~\ref{fp-stabilizer} 
there is a set $F'$ containing $F$ and one other point 
which are the common fixed points for every element of $A$.
As above the set $F'$ is $G$-invariant and we replace $F$ by $F'$, 
so we may assume $F$ contains two points.}

In this case $A$ has index at most two and $\Fix(h) = \Fix(h')$ 
for all $h,h' \in A$.  {From 
Propostion~\ref{rho constant} it follows that $A$ is isomorphic to 
a subgroup of $\T$.}

{
If $h \in G \setminus A$ then $h$ interchanges the two points of $F$ and
reverses the orientation of $H_1(U)$ where $U = S^2 \setminus F.$  It
follows that $\rho_\mu (h^{-1} g h) = -\rho_\mu (g)$ for all $g \in A.$
Also $h$ has two fixed points and $h^2$ fixes the points of $F$ so it has
four fixed points and we conclude that $h^2 = id$. Hence the map
$\phi: \Z/2\Z \to G$ given by $\phi(1) = h$ is a homomorphism so
$G \cong A \rtimes_\phi \Z/2\Z$.
}

\end{proof}

\begin{prop}\label{A pseudo-r}
If $G$ is a  subgroup of $\Symp^\infty_\mu(M)$
which has an infinite normal abelian subgroup $A$ which is a
pseudo-rotation group then $G$ {has an abelian subgroup
of index at most two.}
\end{prop}

\begin{proof}
Since $A$ is non-trivial,   $M$ is $\A, \D^2,$ or $S^2$.
 Since $A$ is infinite, {Lemma~\ref{A-D abelian} and
Lemma~\ref{lem: abelian rot}}
imply there is a set $F$ containing $\X(M) \le 2$ points
such that $F = \Fix(h)$ for all $h \in A.$ 
Let $U = M \setminus F$.  Observe that in all cases
$U$ is an annulus.  
The set $F$ must be invariant under $G$
since $F = \Fix(ghg^{-1}) = g(\Fix(h)) =  g(F)$ for every element $g$ in $G$.   
The subgroup $G_0$ of $G$ that pointwise fixes $F$ has index at most two. Also
elements of $G_0$ leave $U$ invariant and their restrictions to $U$
are isotopic to the identity.

{Since 
the set $\Fix(h)$ is the same for all $h \in A$,
by Proposition~\ref{rho constant} the homomorphism} $\rho_\mu: A \to \T^1$
given by $h \mapsto \rho_\mu(h)$ is injective,
where $\rho_\mu(h)$ is the mean rotation number on $U_c$.
Since conjugating $h$ by an element $g \in G_0$
does not change its mean rotation number we conclude
that $h = ghg^{-1}$ for all $h \in A.$  This means that $G_0$ is contained in
$\Cent^\infty_\mu(A,G),$ the centralizer of $A$ in $G$.

{If $A$ contains elements of arbitrarily large order then  it 
follows from Lemma~\ref{2periodic} that $G_0$ is abelian and we
are done.} So we may assume the order of elements in $A$ is bounded.
Since the order of elements of $A$ is bounded
there are only finitely many possible values for $\rho_\mu(f)$ with
$f \in A.$  Hence, since the assignment $f \mapsto \rho_\mu(f)$ is injective,
we may conclude $A$ is finite in contradiction to our hypotheses.

\end{proof}

\begin{ex}
{
Let $A$ be the subgroup of all rational rotations around the $z$-axis
and let $G$ be the $\Z/2\Z$ extension of $A$ obtained by adding an
involution $g$ which rotates around the $x$-axis and reverses the
orientation of the $z$-axis. More precisely let the homomorphism
$\phi: \Z/2\Z \to \Aut(A)$ be given by $\phi(1) = i_g$ where
$i_g(h) = g^{-1}hg = h^{-1}$. 
Every element of $A$ has finite
order while $A$ itself has infinite order.  Moreover, $G \cong
A \rtimes_\phi (\Z/2\Z)$ is not
abelian even though it has an index two abelian subgroup.}
\end{ex}

{
We are now able to classify up to isomorphism
all pseudo-rotation subgroups  of $\Symp^\infty_\mu(S^2)$ 
which have a non-trivial normal solvable subgroup.  We
denote by $\phi$ the homomorphism $\phi: \Z/2\Z \to \T$ defined
by $\phi(1)  = i$, where $i$ is the automorphism of $\T, \ i(t) = -t$.
}

\begin{prop}\label{char pseudo-r}
If $G$ is a  pseudo-rotation subgroup of $\Symp^\infty_\mu(S^2)$
which has a normal solvable subgroup then $G$ is isomorphic
to either a subgroup of the generalized dihedral group 
$\T \rtimes_\phi (\Z/2\Z)$ or a subgroup of the group $\O$
of orientation preserving symmetries of the regular octahedron
(or equivalently the orientation preserving symmetries of the cube).
\end{prop}

\begin{proof}
Suppose first that there is a $G$-orbit $F_0$ containing two 
points. Then by Proposition~\ref{lem: F finite}, $G$ is isomorphic to
a subgroup of the generalized dihedral group $\T \rtimes_\phi (\Z/2\Z)$ and
we are done.  If this is not the case, then by the same proposition $G$
is finite and every $G$-orbit contains at least three points. 

Let $A$ be the last non-trivial element of the derived series 
of the normal solvable subgroup of $G$, so $A$ is abelian. 
Since $A$ is a characteristic subgroup of that normal solvable subgroup it
is normal in $G$.  Hence for any
$g \in G, \  g(\Fix(A)) = \Fix(g A g^{-1}) = \Fix(A)$ and we observe 
$\Fix(A)$ cannot contain only two points since that is the case we already
considered. Therefore not all elements of $A$ have the same set of fixed
points and we can apply  Lemma~\ref{lem: abelian rot} to conclude the group
$A$ is isomorphic to $\Z/2\Z \oplus \Z/2\Z$.
{We observed in Remark~\ref{rmk: 6pts} that there is an $A$-invariant set
$F$ which is the union of three sets $F_i = \{a_i, b_i\},\  1\le i \le 3$
where $F_i = \Fix(h_i)$ and $h_i$ is one of the non-trivial elements
of $A$. We also noted there that the  elements of 
$A$ preserve the pairs $F_i$, fixing the points
of one of them and switching the points in each of the other pairs.
Since $A$ is normal in $G$, if $h_j = g^{-1}h_ig \in A$ 
then $F_j = \Fix(h_j) = g(\Fix(h_i) = g(F_i)$ so the 
elements of $G$ must permute the three pairs $F_i.$}

We define a homomorphism $\theta: G \to O(3)$ as
follows.  {Given $g \in G$ define a matrix $P = P_g$ by 
$P_{ij} = 1$ if $g(a_j) = a_i$ and 
$g(b_j) = b_i$, $P_{ij} = -1$ if $g(a_j) = b_i$ and
$g(b_j) = a_i$,} 
and $P_{ij} = 0$ otherwise. Each row and column has one non-zero
entry which is $\pm 1$ so $P_g \in O(3).$  It is straightforward to
see that $\theta(g) = P_g$ defines a homomorphism to $O(3)$.  It is
also clear that it is injective since $P_g = I$ implies that $g$
fixes the six points of $F$.  Note that if $h_i$ is a non-trivial
element of $A$, then  $\theta(h_i)$ must be diagonal
since $h_i(F_j) = F_j$, and $\theta(h_i)$ must have two entries
equal to  $-1$ since $h_i$ switches the points of two of the $F_j$'s.
It follows that $\theta(A)$ is precisely the
diagonal entries of $O(3)$ with an even number of $-1$'s.

We need to show that $\theta(G)$ lies in $SO(3)$, i.e., all its
elements have determinant $1$. Clearly
$\theta(A) \subset SO(3)$.  We denote the symmetric group on
three elements by $\cS_3$ and think of it as the permutation matrices
in $O(3)$. We define a homomorphism 
$\bar \theta : G \to \cS_3$ by $\bar \theta(g) = Q \in O(3)$ where
$Q_{ij} = |P_{ij}|$ and $P= \theta(g).$ We observe that $A$ is in the 
kernel of $\bar \theta.$  In fact $A$ equals the 
kernel of $\bar \theta$.  To see this note that if $\bar \theta(g) = I$ 
then $P_g = \theta(g)$ is a diagonal matrix and hence 
$\ker(\bar \theta)$ is abelian.  If $g \in \ker(\bar \theta) \setminus
A$ then the abelian group generated 
by $g$ and $A$ is not $\Z/2\Z \oplus \Z/2\Z$
contradicting Lemma~\ref{lem: abelian rot}. We conclude 
$A = \ker(\bar \theta)$. 

Elements of $\cS_3$ have order one, two or three.  In case 
$\bar \theta(g)$ has order one or three, $g^3$ is in $A$ and hence 
$\det(\theta(g)) = \det(\theta(g))^3 = \det(\theta(g^3))= 1$ 
so $\theta(g) \in SO(3)$.  

If $\bar \theta(g)$ has 
order two and $P_g = \theta(g)$ has determinant $-1$ we argue to a contradiction. 
In this case, without loss of generality we may assume 
\[
P_g =
\begin{pmatrix}
a & 0 & 0\\
0 & 0 & b\\
0 & c & 0\\
\end{pmatrix}
\]
where $\det(P_g) = -abc = -1.$ 
Hence $g^2 = id$. 

If $a =1$ then $bc = 1$ and $P_g^2 = I.$  We consider the $g$ invariant annulus $S^2 \setminus F_1$
and observe that in this annulus $\rho_\mu(g) = 1/2.$  But if $h_1 \in A$ is
the element with $\theta(h_1) = diag(1, -1, -1)$ then $h_1(U) = U$ and  
$\rho_\mu(h_1) = 1/2.$ Since $g$ and $h$ both fix $F_1$ pointwise
they commute by Corollary~\ref{fp-stabilizer} 
and by Proposition~\ref{rho constant}
$\rho_\mu$ is defined and injective on the 
subgroup generated by $g$ and $h_1$.
We conclude $g = h_1$ a contradiction.
Finally suppose $a = -1$ and $bc=-1$.  If $h_2 \in A$ is
the element with $\theta(h_2) = diag(-1, 1, -1)$ then 
\[
\theta(h_2 g) =
\begin{pmatrix}
1 & 0 & 0\\
0 & 0 & b\\
0 & -c & 0\\
\end{pmatrix},
\]
so by the previous argument $\det(\theta(h_2 g)) = 1$. Since
$\det(\theta(h_2)) = 1$ this contradicts the assumption
that  $\det(\theta(g)) = -1$. This completes the proof that 
$\theta(G) \subset SO(3)$.

Since $\theta(G)$ preserves the 
set $\{(\pm 1, 0, 0), (0, \pm 1, 0), (0, 0, \pm 1)\}$ of vertices
of the regular octahedron, it follows that $\theta(G)$ is a subgroup of
the group $\O$ of orientation preserving symmetries of the octahedron.

\end{proof}

\section{Proof of the Main Theorem}\label{main section}

In this section we prove 
\vspace{.1in}

\noindent
{\bf Theorem~(\ref{main thm})}
{\em Suppose $M$ is a compact oriented surface of genus $0$
and  $G$ is a subgroup
 of $\Symp^\infty_\mu(M)$ which has full support of finite type, 
e.g. a subgroup of $\Symp^\omega_\mu(M)$.  Suppose further that $G$ has an infinite normal solvable subgroup. Then $G$ is virtually abelian.
}

\vspace{.1in}

Throughout this section $M$ will denote 
a compact oriented surface of genus $0$, perhaps with boundary.
We begin with a lemma that allows us to replace the hypothesis that $G$ has an infinite normal solvable subgroup with the hypothesis that $G$ has an infinite normal abelian subgroup.

\begin{lemma}\label{solv to abel}
If $G$ is a  subgroup of  $\Symp^\infty_\mu(M)$,  which contains an
infinite normal solvable subgroup, then $G$ contains a finite
index subgroup which has an infinite normal abelian subgroup.
\end{lemma}

\begin{proof}
Let $N$ be the infinite normal solvable subgroup.  The proof is by
induction on the length $k$ of the derived series of $N$.  If $k=0$
then $N$ is abelian and the result holds.  Assume the result holds for
$k \le k_0$ for some $k_0 \ge 0$ and suppose the length of the derived
series for $N$ is $k = k_0 + 1$.  Let $A$ be the abelian group which
is the last non-trivial term in the derived series of $N$.  The group
$A$ is invariant under any automorphism of $N$ and hence is normal in
$G$. If $A$ is infinite we are done.

We may therefore assume $A$ is finite  and hence that  in a suitable
metric $A$ is a group of isometries.  No non-trivial orientation preserving
isometry of $M$ can have infinitely many fixed points so we 
let $F$ be the finite set  of fixed
points of non-trivial elements of $A$. Since $A$ is normal $g(F) = F$
for any $g \in G.$   Let $G_0$ be the {normal finite index subgroup} 
of $G$ which pointwise fixes $F.$   

{Let $N_1 = G_0 \cap N$ so $N_1$ is an infinite solvable normal  
subgroup of $G_0$}
and the $k$-th term $A_1$ in the derived series of $N_1$ is contained
in $A \cap G_0$.  We claim that $A_1$ is trivial.  If $F$ contains more
than $\X(M)$ points  then this follows from the fact that a non-trivial
isometry isotopic to the identity fixes exactly $\X(M)$ points since
each fixed point must have Lefschetz index $+1$.  

 Otherwise $F$ contains $\X(M)$ points and we may
blow up these points to form an annulus $\A$ on which $N_1$ acts
preserving each boundary component.  Since each element of $A_1$ is a
commutator of elements in $N_1$, it acts on $\A$ with mean rotation
number zero. Each element of $A_1$ therefore contains a fixed point in
the interior of $\A$ by Lemma~\ref{int rho}.  {However, a 
finite order isometry of the annulus which is isotopic to the identity
and contains an interior fixed point must be the identity. This
is because every fixed point must have Lefschetz index $+1$ and the Euler
characteristic of $\A$ is $0.$}  This shows that $A_1$ acts trivially
on $\A$ and hence on $M$. This verifies the 
claim {that $A_1$ is trivial} and hence
the length of the derived series for $N_1$ is at most $k_0$.  The
inductive hypothesis completes the proof.
\end{proof}


The next lemma states the condition we use to prove that   $G$ is virtually abelian.

\begin{lemma}  \label{ending lemma} Suppose $G_0$ is a subgroup of $\Symp^\infty_\mu(M)$ which has full support of finite type.  Suppose further that there is an infinite family of disjoint  $G_0$-invariant open annuli.  Then $G_0$ is abelian.
\end{lemma}

\proof We assume that $[G_0,G_0]$ contains a non-trivial element $h$
and argue to a contradiction. For each $G_0$-invariant open annulus
$V$, let $V_c$ be its annular compactification and let $h_c :V_c \to
V_c$ be the extension of $h|_V$ 
{(see Definition~\ref{defn: annular comp}
or Definition~2.7 of \cite{FH-ent0} for details). 
Since $h$ is a commutator of
elements of $G_0$ and since $G_0$ extends to an action on $V_c$, $h_c$
is a commutator and so has mean rotation number zero.  Lemma~\ref{int
  rho} therefore implies that $\Fix(h) \cap V \ne \emptyset$.

By assumption, $\Fix(h)$ does not contain any open set and so
$\Fix(h|_V)$ is a proper subset of $V$.  We claim that $\fr(V) \cup
\Fix(h|_V)$ is not connected.  To see this, let $S$ be the end
compactification of $V$ obtained by adding two points, one for each
end of $V$ and let $h_S : S \to S$ be the extension of $h|_V$ that
fixes the two added points.  If $\fr(V) \cup \Fix(h|_V)$ is connected
then each component $W$ of $S \setminus \Fix(h|_S)$ is an open disk.
A result of Brown and Kister \cite{brnkist} asserts $W$ is
$h_S$-invariant.  However, then the Brouwer plane translation theorem
would imply that $h$ has a fixed point in $W$.  This contradiction
proves the claim.

By passing to a subfamily of the $G_0$-invariant annuli, we may assume
that the following is either true for all $V$ or is false for all $V$:
some component of $\Fix(h)$ intersects both components of the frontier
of $V$.  In the former case, the interior of each $V$ contains a
component of $\Fix(h)$.  In the latter case, no component of $\Fix(h)$
intersects more than two of the annuli in our infinite family.  In
both cases, $\Fix(h)$ has infinitely many components in contradiction
to the assumptions that $h$ is non-trivial and that $G_0$ has full
support of finite type.
\endproof

We need an elementary topology result.

\begin{lemma} \label{breve}  Suppose that $C\subset \Int(M)$ is a closed connected set which is nowhere dense and has two complementary components $U_1$ and $U_2$.    Then $C' = \fr(U_1) \cap\ \fr(U_2)$ is a closed connected set with two complementary components, {$U_1'$ and $U_2'$}, each of which has frontier $C'$ and is equal to the  interior of its closure. Moreover $U_i$ is dense
in $U_i'$.
\end{lemma} 

\proof To see that $C'$ separates $U_1$ and $U_2$, suppose that
$\sigma$ is a closed path in $M \setminus C'$ with initial endpoint
in $U_1$ and terminal endpoint in $U_2$.  Then $\sigma \cap \fr(U_2)
\ne \emptyset$ and we let $\sigma_0$ be the shortest initial segment
of $\sigma$ terminating at a point $x \in \sigma \cap \fr(U_2)$.  Each
$y\in \sigma_0 \setminus x$ has a neighborhood that is disjoint from
$U_2$.  Since $C$ has empty interior, every neighborhood of $y$ must
intersects $U_1$.  It follows that every neighborhood of $x$ must
intersect $U_1$ and hence that $x \in C'$.  This contradiction proves
that $C'$ separates $U_1$ and $U_2$.  Since   the union of
  $U_1$ and $U_2$ is dense in $M$, the components $V_1$ and $V_2$
of $M \setminus C'$ that contain them are the only components of
$M \setminus C'$.  Every neighborhood of a point in $C'$ intersects
both $U_1$ and $U_2$ and so intersects both $V_1$ and $V_2$.  Thus $C'
\subset \fr(V_1) \cap \fr(V_2)$.  Since $\fr(V_1), \fr(V_2) \subset
C'$ we have that $C' \subset \fr(V_1) \cap \fr(V_2) \subset \fr(V_i)
\subset C'$ and hence $C' = fr(V_1) = fr(V_2)$ 
which implies that $V_1$ and $V_2$ are the interior of
their closures. Since $C$ is nowhere dense in $M$ it follows
that  $U_1 \cup U_2$ is dense in $M$ and hence that 
$U_i$ is dense in $U_i'$.
\endproof

\begin{notn}  If $C$ and $C'$ are as in Lemma~\ref{breve} then we say that $C'$ is obtained from $C$ by {\em trimming}.
Recall that if $f \in \zz(M)$ and $U \in \cA_f$, then the
rotation number function $\rho = \rho_{f|_U}: U \to S^1$ is well defined
and continuous.  A component of a point pre-image  of $\rho$ is {\em a   level set for $\rho$} and is said to be {\em irrational} if its
$\rho$-image is irrational and to be {\em interior} if it is disjoint
from the frontier of $U$.  If $C$ is an interior level set which is nowhere
dense and $C'$ is obtained from $C$ by trimming
then we will call $C'$ a {\em trimmed level set}.  The collection of all trimmed
level sets for $f$ will be denoted $\C(f).$
\end{notn}

\begin{lemma} \label{disjoint or equal} Suppose $G$ is a
subgroup of $\Symp^\infty_\mu(M)$ containing an abelian normal
subgroup $A$, that $f \in \zz(M)$ lies in $A$ and that $U \in
\cA_f$.     Suppose further that  $C_1'$ and $C_2'$ are obtained from nowhere dense irrational interior level sets $C_1$ and $C_2$  for $\rho = \rho_{f|_U}$ by trimming.  Letting $B_i$ be the  component of $M\setminus (C_1'\cup C_2')$ with frontier $C'_i$ and $V$ be the component of $M\setminus (C_1'\cup C_2')$, 
with frontier $(C_1'\cup C_2')$, assume that
 \begin{enumerate}
\item $\mu(V) < \mu(B_1)<  \mu(B_2).$ \label{item:V is small} 
\item   $\mu(B_2 ) > \mu(M)/2.$ \label{item:more than half}  
\end{enumerate} 
Then for all $g \in G$,     either $g(V) \cap V= \emptyset$  or $g( V) = V$.   
In particular, there is a finite index subgroup of $G$ that preserves $V$.
\end{lemma}

\begin{proof} 
We assume that $g(V) \cap V\ne \emptyset$ and $g( V) \ne V$ and argue
to a contradiction.  Since $A$ is normal and abelian, $h = g^{-1}fg$
commutes with $f$.  Corollary~\ref{Z centralizer} part (1) implies
that $h$ permutes the elements of $\cA_f$ and hence that $h^n(U) = U$
for some $n \ge 1$. It then follows from Corollary~\ref{Z centralizer}
part (2) that $h^n$ preserves each level set 
for $\rho = \rho_{f|_U}$, and hence
each trimmed level set for $\rho$, and so preserves $V, B_1$ and
$B_2$.  Equivalently, $g(V), g(B_1)$ and $g(B_2)$ are $f^n$-invariant.

Since $g(V) \cap V \ne \emptyset$ and $g(V) \ne V$, there is a
component $W$ of $g(V) \cap V$ such that $\fr(W) \cap \ \fr(V)\ne
\emptyset$.  To see this observe that otherwise one of
$V$ and $g(V)$ would properly contain the other
 which would contradict the fact that $g$ preserves $\mu.$ 
Since $f^n$ preserves both $V$ and $g(V)$, it
permutes the components of their intersection.  Since $f$ preserves
area, $f^m(W) = W$ for some $m>0.$ If every simple closed curve in $W$
is inessential in $V$ then $\rho$ has a constant rational value
on a dense subset of $W$ by Lemma~\ref{lem: periodic}.
This contradicts the fact that $\rho$ is continuous on $U$ and
takes only irrational values on $\fr(V)$. We conclude that there is a
simple closed curve $\alpha \subset W$ which is essential in
$V.$  Let $\beta = g^{-1}(\alpha) \subset V$.
 Item~\pref{item:V is small} implies that
$\beta$ is also essential in $V$ because if it were inessential
the component of its complement which lies in $V$ would contain
either $g^{-1}(B_1)$ or $g^{-1}(B_2).$ 

Let $B_i'$ be the subsurface of $M$ bounded by $\beta$ that contains
$B_i$.  Item \pref{item:more than half} rules out the possibility that
$B_2 \subset g(B_1')$ so it must be that $B_1 \subset g(B_1')$ and
$B_2 \subset g(B_2')$.  Item \pref{item:V is small} therefore implies
that $B_i \cap g(B_i) \ne \emptyset$.  If there is a simple closed
curve in $B_i$ whose $g$-image is in $V$ and essential in $V$ then there is a
proper subsurface of $B_i$ whose $g$ image contains 
$B_i$ in contradiction to the
fact that $g$ preserves area.  It follows that either $g(B_i) = B_i$
 or there is a component  $W_i'$ of $g(B_i) \cap V$ which
is inessential in $V$ and whose
frontier intersects $C_i'$.  As above, this implies that $\rho$ is
constant and rational on $W'$ in contradiction to the fact that $\rho$
is continuous on $U$ and takes an irrational value on $C_i'$.
This contradiction completes the proof.
\end{proof}

 We are now prepared to complete the proof of our main theorem.
\bigskip

\noindent
{\bf Theorem~(\ref{main thm})}
{\em 
Suppose $M$ is a compact oriented surface of genus $0$
and  $G$ is a subgroup
of $\Symp^\infty_\mu(M)$ which has full support of finite type,
e.g., a subgroup of $\Symp^\omega_\mu(M)$.  Suppose further that $G$ has 
an infinite normal solvable subgroup. Then $G$ is virtually abelian.
}

\begin{proof}
 By Lemma~\ref{solv to abel} it suffices to prove the result when $G$
has an infinite normal abelian subgroup $A$. 
If $A$ contains an element of
positive entropy the result follows by Proposition~\ref{A +entropy}.
If the group $A$ is a pseudo-rotation group the result follows from 
Proposition~\ref{A pseudo-r}. 

We are left with the case that $A$ contains an element $f \in \zz(M)$.
Let $U$ be an element of $\cA_f$ and let $\rho = \rho_{f|_U}$.  Since there
are only countably many level sets of $\rho$ with positive measure, all but
countably many have empty interior or equivalently are
nowhere dense.  Choose a nowhere dense interior irrational level set
$C$ of $\rho$.  One component of $M \setminus C$, say, $Y$, will be an
open subsurface with $\mu(Y) \le \mu(M)/2.$ Choose a nowhere dense
interior irrational level set $C' \subset Y \cap U$ and let $Y'$ be
 the open subsurface which is the component 
of the complement of $C'$ satisfying $Y' \subset Y$.
Finally, choose nowhere dense interior irrational level sets $\hat
C_1, \hat C_2 \subset (Y \setminus \cl(Y'))$ so that the annulus $\hat
V$ bounded by the trimmed sets $\hat C_1'$ and $\hat C_2'$ has measure
less than the measure of $Y'$.  Then $\hat C_1'$ and $\hat C_2'$
satisfy the hypotheses of Lemma~\ref{disjoint or equal}.  It follows
that the subgroup $G_0$ of $G$ that preserves $\hat V$ has finite
index in $G$.

Note that any two trimmed irrational level sets $C_1'$ and $C_2'$ in
$\hat V$ satisfy the hypotheses of Lemma~\ref{disjoint or equal}.
Moreover if $V$ is the essential open subannulus bounded by $C_1'$ and
$C_2'$ then $g(V) \cap V \ne \emptyset$ for each $g \in G_0$ because
each such $g$ preserves $\hat V$ and preserves area.  Lemma~\ref{disjoint
  or equal} therefore implies that each such $V$ is $G_0$-invariant so
Lemma~\ref{ending lemma} completes the proof.
\end{proof}

\section{The Tits Alternative}

We assume throughout this section that $M$ be a compact oriented surface of genus $0$, ultimately proving a special case (Theorem~\ref{tits special case}) of the Tits Alternative.


\begin{lemma}\label{abelian case} 
Suppose $f \in \cZ(M)$ and $U \in \cA(f)$ is a 
maximal annulus for $f.$  If $G$ is a subgroup of $\Symp^\omega_\mu(M)$ 
whose elements preserve each element of an infinite family
of  trimmed level sets for $\rho_{f|_U}$
lying in $U$, then $G$ is abelian.
\end{lemma}

\begin{proof}
Since infinitely many of the trimmed level sets in $U$ are preserved
by $G$ so is each open annulus bounded by two such level sets.  It is
thus clear that one may choose infinitely many disjoint open annuli in
$U$ each of which is $G$-invariant.  The result now follows from
Lemma~\ref{ending lemma}.
\end{proof}

\begin{lemma} \label{trans num} 
Suppose $f: \A \to \A$ is a homeomorphism isotopic to the identity
with universal covering lift $\ti f: \ti \A \to \ti \A$ which has a 
non-trivial translation interval and suppose $\alpha \subset \A$ is an
embedded arc joining the two boundary components of $\A$.
Then there exists $m >0$ such that 
for any $\gamma \subset \A$ which is 
an embedded arc, disjoint from $\alpha$, 
joining the two boundary components of $\A$,
every  lift of $f^{k}(\gamma),\ |k| > m$ must intersect more than
one lift of $\alpha.$

\end{lemma}

\begin{proof}

Let $\ti \alpha$ be a lift of $\alpha$  and let  $T$ 
be a generator for the group of covering translations of $\ti \A$. We claim that
there is an $m >0$ such that for all $k \ge m,$
  $\ti f^k(\ti \alpha)$   intersects
at least $6$ adjacent lifts of $\alpha$.  
To show the claim consider the fundamental domain $D_0$ 
bounded by $\ti \alpha$ and $T(\ti \alpha).$
If there are arbitrarily large $k$ for which  $\ti f^k(\ti \alpha)$ 
intersects fewer than six translates
of $\ti \alpha,$ then for such $k, \ \ti f^k(D_0)$ 
would lie in an adjacent strip  of five adjacent translates of $D_0$
and the translation interval for $\ti f$ would have length $0$.  This contradiction verifies the claim.

There is a lift  $\ti \gamma$ 
of $\gamma$ that is contained in $D_0$. Let $X$ be the  region   bounded by $T^{-1}(\ti \gamma)$ and $T(\gamma)$ and note that $\ti \alpha \subset X$.  If the lemma fails then we can choose   $k \ge m$ and a lift $\ti h $ of $f^k$ such that $\ti h(\ti \gamma)$  is disjoint from $T^i(\ti \alpha)$ for all $i \ne 0$.  It follows that   $\ti h(X)$ is contained in $\cup_{j=-2}^2 T^{j}(D_0)$ in contradiction to the fact that  $\ti h(\ti \gamma)$ intersects at least six translates of $\ti \alpha$. 
\end{proof}

\begin{lemma} \label{free case} 
{Suppose $f,g \in \cZ(M)$ and there are trimmed level sets
$C_1' \in \C(f)$ and $C_2' \in \C(g)$ such that  there exist points
$x_i \in M \setminus (C_1' \cup C_2'),\ 1 \le i \le 4$ with the
following properties:
\begin{enumerate}
\item $x_1,x_2$ lie in the same component of the complement of $C_1'$
and $x_3,x_4$ lie in the other component of this complement.
\item $x_1,x_3$ lie in the same component of the complement of $C_2'$
and $x_2,x_4$ lie in the other component of this complement.
\end{enumerate}
}
Then for some $n>0,$ the diffeomorphisms $f^n$ and $g^n$ generate
a non-abelian free group.
\end{lemma}

\begin{figure}[ht!]
\centering
\includegraphics[height=3.0in]{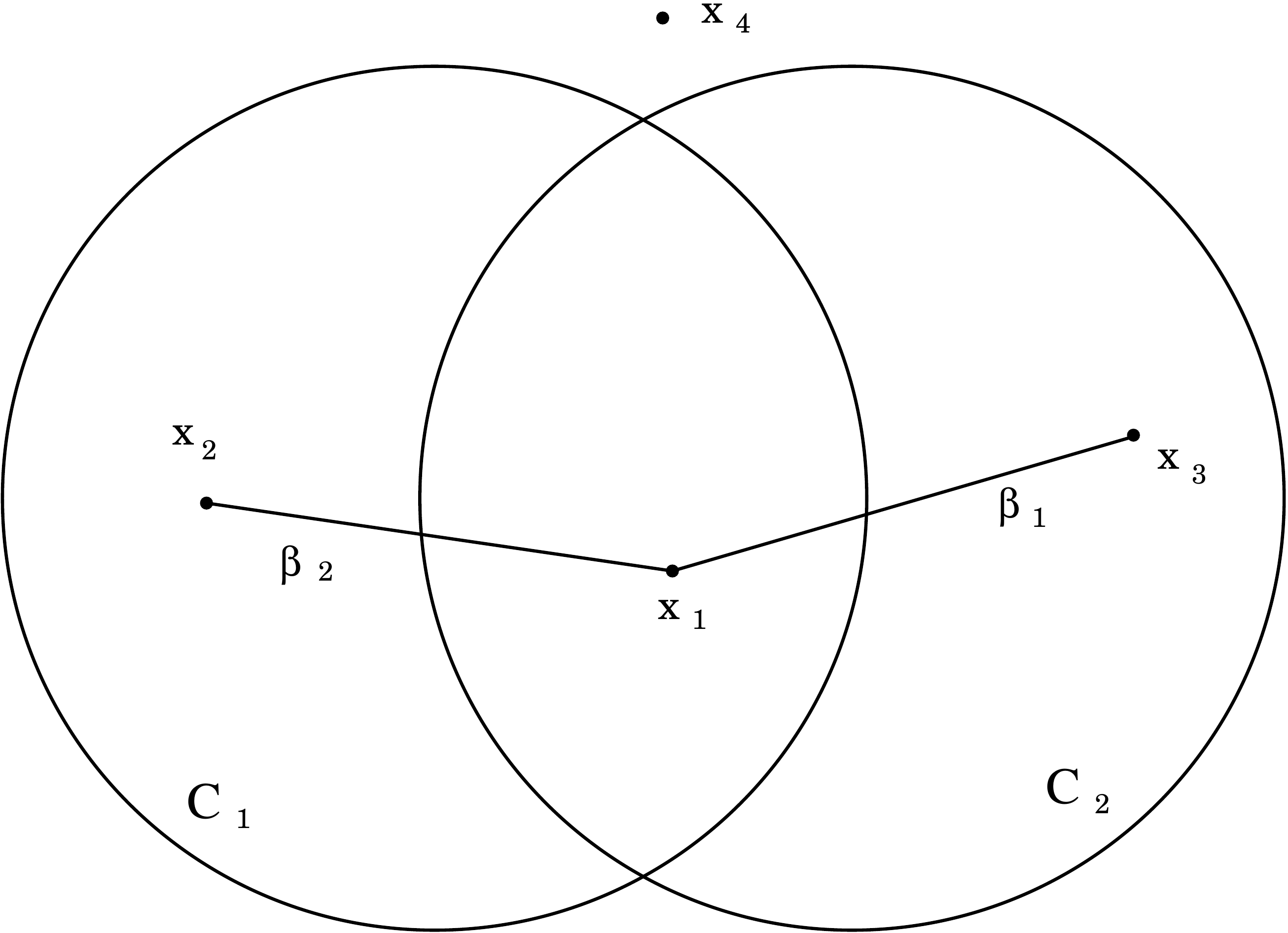}
\caption{The curves $\beta_1$ and $\beta_2$}
\label{fig1}
\end{figure}

\begin{proof}
Let $C_i$ be the untrimmed level set whose trimmed version is
$C_i'$.  We claim that by modifying the points $\{ x_i\}$ slightly
we may assume that the hypotheses (1) and (2) hold with $C_i'$ replaced
by $C_i$. This is because each component of the complement of 
$C_i$ is a dense open subset of a component of the complement of
$C_i'$.  Hence each $x_j$ can be perturbed slightly to $\hat x_j$ which,
for each $i$,  is in the
component of complement of $C_i$ which is a dense open subset 
of the component of the complement of $C_i'$ containing $x_j$.  Henceforth
we will work with $C_i$ and refer to $\hat x_j$ simply as $x_j$.

Let $\beta_1$ be a path in $M \setminus  C_2$ joining $x_1$ and $x_3$;
so $\beta_1$ crosses $ C_1$ and is disjoint from $ C_2.$ 
Likewise let $\beta_2$ be a path in $M \setminus  C_1$ joining $x_1$ and $x_2$;
so $\beta_2$ crosses $ C_2$ and is disjoint from $ C_1.$  (See Figure \ref{fig1}.)
The level set $ C_1$ is an intersection 
\[
 C_1 = \bigcap_{n=1}^\infty \cl(B_n),
\]
where each $B_n$ is an $f$-invariant open annulus with
$\cl(B_{n+1}) \subset B_n$ and the rotation interval 
$\rho(f, B_n)$ of the annular compactification $f_c$ of
$f|_{B_n}$ is non-trivial (see section 15 and the proof of 
Theorem~ 1.4 in \cite{FH-ent0}).  For $n$ sufficiently large
$\cl(B_n)$ is disjoint from $\beta_2$ and $\{x_i\}_{i=1}^4.$
Let $A_1$ be a choice of $B_n$ with this property.

We may choose a closed subarc $\alpha_1$ of $\beta_1$ whose
interior lies in $A_1$ and whose endpoints are in different
components of the complement of $A_1$.   We will use intersection
number with $\alpha_1$ with a curve in $A_1$ to get a lower
bound on the number of times that curve ``goes around'' the
annulus $A_1.$

\begin{figure}[ht!]
\centering
\includegraphics[height=3.5in]{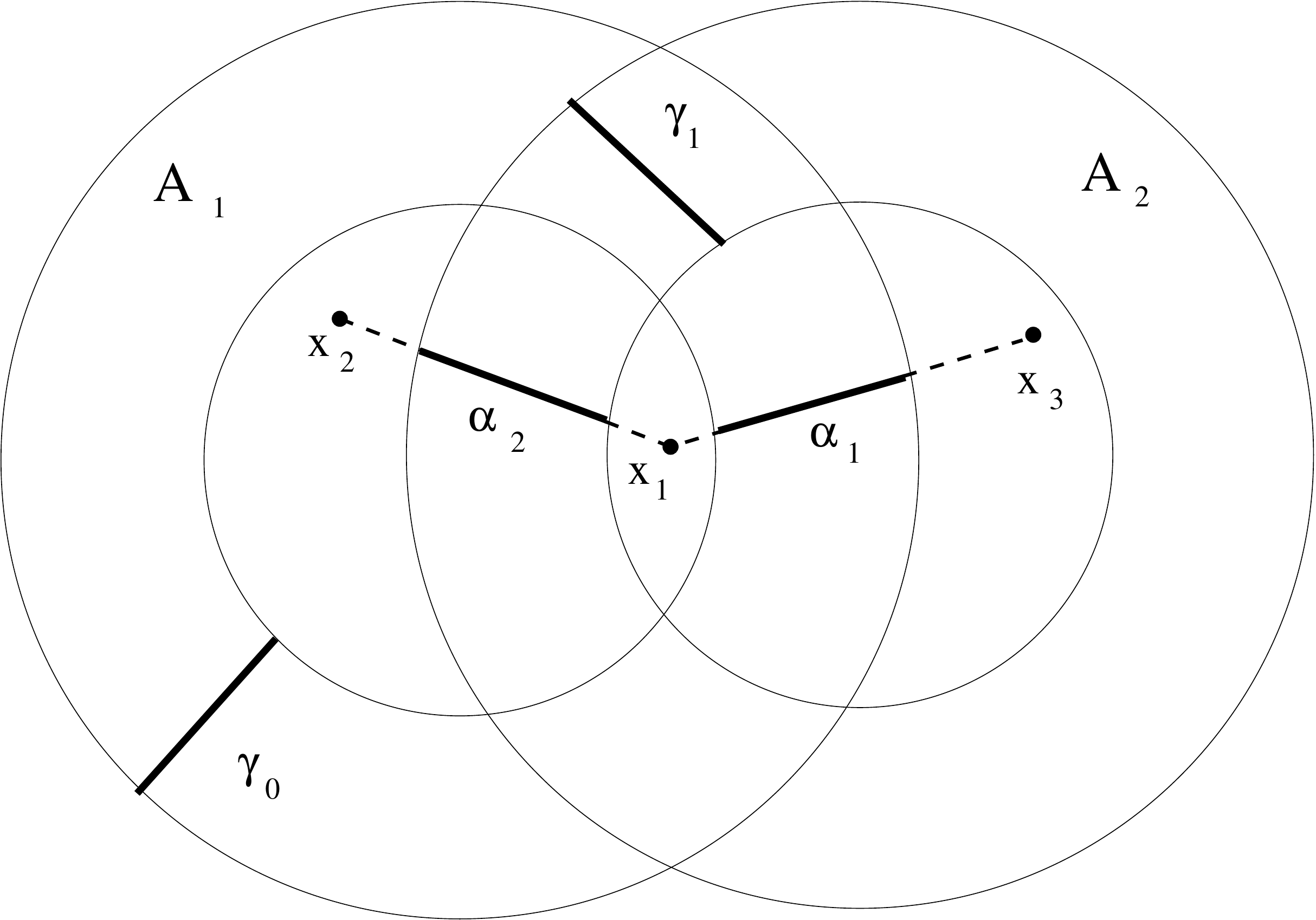}
\caption{The curves $\alpha_1,\ \alpha_2,\ \gamma_0$ and $\gamma_1$}
\label{fig2}
\end{figure}

Similarly the level set $ C_2$ is an intersection 
\[
 C_2 = \bigcap_{n=1}^\infty \cl(B_n'),
\]
where each $B_n'$ is a $g$-invariant open annulus with
$\cl(B_{n+1}') \subset B_n'$ and the rotation interval 
$\rho(g, B_n')$ of the annular compactification $g_c$ of
$f|_{B_n'}$ is non-trivial.
We construct $A_2$ and the arc $\alpha_2$ in a fashion
analogous to the construction of $A_1$ and $\alpha_1$. 
By construction $\alpha_1$ is disjoint from $A_2$ and crosses
$A_1$ while $\alpha_2$ is disjoint from $A_1$ and crosses
$A_2.$ (See Figure \ref{fig2}.)

Note that any essential closed curve in $A_1$ must intersect $\alpha_1$ and
must contain points of both components of the complement of $ C_2.$
To see this latter fact we note that we constructed $\alpha_1$
to lie in one component of the complement of $ C_2$ but we could as
well have constructed $\alpha_1'$ in the other component of this
complement.  Any essential curve in $A_1$ must intersect both
$\alpha_1$ and $\alpha_1'.$
Similarly any essential curve in $A_2$ must intersect $\alpha_2$ and
must contain points of both components of the complement of $ C_1$.

There is a key consequence of these facts which we now explore.  Let
$\gamma_0$  be an arc with interior in $A_1$, disjoint from $A_2
\cup \alpha_1$ and with endpoints in different components of the
complement of $A_1$.  Replace $f$ and $g$ by $f^m$ and $g^{m'}$ where
$m$ and $m'$ are the numbers guaranteed by Lemma~\ref{trans num}
for $f$ and $g$ respectively.
Then we know that for $k \ne 0$ the curve $f^k(\gamma_0)$ 
must intersect more than one lift of the arc
$\alpha_1$ in the universal covering $\ti A_1$.

It follows that
$f^k(\gamma_0)$ contains a subarc whose union with a subarc of
$\alpha_1$ is essential in $A_1$.  Hence we conclude that $f^k(\gamma_0)$
contains a subarc crossing $A_2$, i.e. a subarc $\gamma_1$ 
whose interior lies in $A_1 \cap A_2$
(and hence is disjoint from $\alpha_2$), and whose endpoints are in
different components of the complement of $A_2.$ (See Figure~\ref{fig2}.)
 Since we replaced $g$ by  $g^{m'}$ above
we know that for $k \ne 0$ the curve $g^k(\gamma_1)$ 
must intersect more than one lift of the arc
$\alpha_2$ in the universal covering $\ti A_2$.

We can now construct $\gamma_2$ in a similar manner but switching
the roles of $f$ and  $g, \ \alpha_1$ and $\alpha_2,$ and $A_1$
and $A_2$.  More precisely, for any $k \ne 0$ the curve
$g^k(\gamma_1)$ contains a subarc whose
union with a subarc of $\alpha_2$ is essential in $A_2$.  It follows
that $g^k(\gamma_1)$ contains a subarc $\gamma_2$ 
whose interior lies in $A_1 \cap A_2$
and whose endpoints are in different components of the complement
of $A_1.$ Note that $\gamma_2$ is a subarc of $g^mf^k(\gamma_0).$

 We can repeat this construction indefinitely, each time switching 
$f$ and $g$.  Hence if we are given 
$h = g^{m_1}f^{k_1} \dots g^{m_n}f^{k_n}$ and $m_i \ne 0,\ k_i \ne 0$
we can obtain a non-trivial arc $\gamma_{2n}$
which is a subarc of $h(\gamma_0)$.  Since $\gamma_{2n} \subset
A_1 \cap A_2$ and $\gamma_0 \cap A_2 = \emptyset$ it is not possible
that $h = id.$  But every element of the group generated by $f$ and
$g$ is either conjugate to a power of $f$, a power of $g,$ or
an element expressible in the form of $h$.  Hence we conclude
that the group generated by $f$ and $g$ is a non-abelian free group.
\end{proof}

\noindent{\bf Proof of Theorem~\ref{tits special case}.}
Suppose $f \in G \cap \zz(M)$ and $U \in \cA(f)$ is a maximal annulus
for $f$. One possibility is that there is a finite index subgroup
$G_0$ of $G$ which preserves infinitely many of the trimmed rotational
level sets for $f$ which lie in $U$.  In this case Lemma~\ref{abelian case}
implies $G_0$ is abelian and we are done.

If this possibility does not occur,
we claim that there exists a trimmed level set $C$ in $U$ 
and $h_0 \in G$ such
that $h_0(C) \cap C \ne C$ but $h_0(C) \cap C \ne \emptyset$. 
If this is not the case then for every $h \in G$ and every 
$C$ either $h(C) = C$ or $h(C) \cap C = \emptyset$.
It follows that the $G$-orbit of $C$ consists of pairwise disjoint
copies of $C$.  Since elements of $G$ preserve area this orbit
must be finite and it follows that the subgroup of $G$ which
stabilizes $C$ has finite index.  If we now choose $C_0$ and $C_0'$,
two trimmed level sets in $U$ and let $G_0$ be the finite index subgroup of
$G$ which stabilizes both of them, then the annulus $U_0$ which
they bound is $G_0$-invariant. Now if $C$ is any trimmed level set in $U_0$ then
its $G_0$-orbit lies in $U$ and we conclude from area preservation
that $g(C) \cap C \ne \emptyset$ for all $g \in G_0.$   If for
some $g$ and $C$ we have  $g(C) \cap C \ne C$ we have demonstrated the
claim and otherwise we are in the previous case since $G_0$ preserves
the infinitely many trimmed level sets lying in $U_0$.

So we may assume that the claim holds, i.e., that
 there is $h_0 \in G$  and a trimmed level set $C_1$ for $f$ in $U$ such
that $h_0(C_1) \cap C_1 \ne C_1$ and $h_0(C_1) \cap C_1 \ne \emptyset$.
Let $C_2 = h_0(C_1)$ and $g = h_0fh_0^{-1}$.   Since $C_1 \ne C_2$
and each is the common frontier of its complementary components, it
follows that each of the complementary components of $C_1$ has non-empty
intersection with each of the complementary components of $C_2$.
Hence we may choose points 
$x_i \in M \setminus (C_1 \cup C_2),\ 1 \le i \le 4$ satisfying the
hypothesis of Lemma~\ref{free case} which completes the proof.
\qed

\bibliographystyle{plain}

\begin{thebibliography}{10}

\bibitem{bfh:tits1}
Mladen Bestvina, Mark Feighn, and Michael Handel.
\newblock The {T}its alternative for {${\rm Out}(F\sb n)$}. {I}. {D}ynamics of
  exponentially-growing automorphisms.
\newblock {\em Ann. of Math. (2)}, 151(2):517--623, 2000.

\bibitem{bfh:tits2}
Mladen Bestvina, Mark Feighn, and Michael Handel.
\newblock The {T}its alternative for {${\rm Out}(F_n)$}. {II}. {A} {K}olchin
  type theorem.
\newblock {\em Ann. of Math. (2)}, 161(1):1--59, 2005.

\bibitem{B-M}
Edward Bierstone and Pierre~D. Milman.
\newblock Semianalytic and subanalytic sets.
\newblock {\em Inst. Hautes \'Etudes Sci. Publ. Math.}, (67):5--42, 1988.

\bibitem{brnkist}
M.~Brown and J.~M. Kister.
\newblock Invariance of complementary domains of a fixed point set.
\newblock {\em Proc. Amer. Math. Soc.}, 91(3):503--504, 1984.

\bibitem{Farb-Shalen}
Benson Farb and Peter Shalen.
\newblock Groups of real-analytic diffeomorphisms of the circle.
\newblock {\em Ergodic Theory Dynam. Systems}, 22(3):835--844, 2002.

\bibitem{FH-ent0}
J.~Franks and M.~Handel.
\newblock Entropy zero area preserving diffeomorphisms of $S^2$
\newblock {\em Geometry \& Topology} 16:2187--2284, 2012.

\bibitem{F-Poincare}
John Franks.
\newblock Generalizations of the {P}oincar\'e-{B}irkhoff theorem.
\newblock {\em Ann. of Math. (2)}, 128(1):139--151, 1988.

\bibitem{franks:recurrence}
John Franks.
\newblock Recurrence and fixed points of surface homeomorphisms.
\newblock {\em Ergodic Theory Dynam. Systems}, 8$^*$(Charles Conley Memorial
  Issue):99--107, 1988.

\bibitem{Iv1}
Nikolai~V. Ivanov.
\newblock {\em Subgroups of {T}eichm\"uller modular groups}, volume 115 of {\em
  Translations of Mathematical Monographs}.
\newblock American Mathematical Society, Providence, RI, 1992.
\newblock Translated from the Russian by E. J. F. Primrose and revised by the
  author.

\bibitem{katok:horseshoe}
A.~Katok.
\newblock Lyapunov exponents, entropy and periodic orbits for diffeomorphisms.
\newblock {\em Inst. Hautes \'Etudes Sci. Publ. Math.}, (51):137--173, 1980.

\bibitem{Katok2}
A.~Katok.
\newblock Hyperbolic measures and commuting maps in low dimension.
\newblock {\em Discrete Contin. Dynam. Systems}, 2(3):397--411, 1996.

\bibitem{mccarthy:tits}
John McCarthy.
\newblock A ``{T}its-alternative'' for subgroups of surface mapping class
  groups.
\newblock {\em Trans. Amer. Math. Soc.}, 291(2):583--612, 1985.

\bibitem{simon:index}
Carl~P. Simon.
\newblock A bound for the fixed-point index of an area-preserving map with
  applications to mechanics.
\newblock {\em Invent. Math.}, 26:187--200, 1974.



\bibitem{smale:DDS}
S.~Smale
\newblock{ Differentiable dynamical systems}
\newblock{\em Bull. Amer. Math. Soc.} 73:747--817, 1967.


\bibitem{Sternberg}
Shlomo Sternberg.
\newblock Local {$C^{n}$} transformations of the real line.
\newblock {\em Duke Math. J.}, 24:97--102, 1957.

\bibitem{Tits}
J.~Tits.
\newblock Free subgroups in linear groups.
\newblock {\em J. Algebra}, 20:250--270, 1972.

\end{thebibliography}

\end{document}